\documentclass[reqno,preprint,aos,authoryear]{imsart}

\RequirePackage{amsthm,amsmath,amssymb,natbib}
\RequirePackage[OT1]{fontenc}
\RequirePackage[colorlinks,citecolor=blue,urlcolor=blue]{hyperref}
\usepackage{booktabs}
\usepackage{nonfloat}
\usepackage {anysize,amsmath,amssymb,mathtools}
\usepackage{pgf}

\startlocaldefs
\numberwithin{equation}{section}
\theoremstyle{plain}
\newtheorem{lemma}{Lemma}[section]

\newtheorem{theorem}{Theorem}[section]
\newtheorem{corollary}{Corollary}[section]

\def\journal@name{}
\endlocaldefs

\newcommand{\ul}[1]{\underline{#1}}

\newcommand{\norm}[1]{\left\Vert#1\right\Vert}
\newcommand{\abs}[1]{\left\vert#1\right\vert}

\marginsize {1in} {1in} {1in} {1in}

\begin{document}                                                                                                                                                                     

\begin{frontmatter}
\title{Adaptive Bayesian Estimation of Conditional Densities
%\protect\thanksref{T1}
}
\runtitle{
Adaptive Bayesian Estimation of Conditional Densities}
\thankstext{T1}{First version: May 2014,  current version: \today.}
\thankstext{T2}{We thank the Co-editor and referees for helpful comments.}

\begin{aug}
\author{\fnms{Andriy} \snm{Norets}\thanksref{t1}
%\thanksref{t1}
\ead[label=e1]{andriy\_norets@brown.edu}
%\ead[label=u1,url]{http://www.econ.uiuc.edu/\~{ }anorets}
}
\and
\author{\fnms{Debdeep} \snm{Pati}
\thanksref{t2}
\ead[label=e2]{debdeep@stat.fsu.edu}
 %\ead[label=u2,url]{http://www.princeton.edu/\~{ }jpelenis} 
}

\thankstext{t1}{Associate Professor, Department of Economics, Brown University
}

\thankstext{t2}{ 
Assistant Professor, Department of Statistics, Florida State University
}

\runauthor{A. Norets and D. Pati}

\affiliation{Brown University and Florida State University}

\address{Economics Department, \\
Brown University,
Providence, RI 02912
\\
\printead{e1}
%\\
%\phantom{E-mail:\ }
%\printead*{e2}
}
\address{ \\ Department of Statistics,\\
Florida State University 
\\
\printead{e2}
%\\
%\phantom{E-mail:\ }
%\printead*{e2}
}

\end{aug}

\begin{abstract}

We consider a non-parametric Bayesian model for conditional densities.  The model is a finite mixture of normal distributions with covariate dependent multinomial logit mixing probabilities.  A prior for the number of mixture components is specified on positive integers.  The marginal distribution of covariates is not modeled.  We study asymptotic frequentist behavior of the posterior in this model.  Specifically, we show that when the true conditional density has a certain smoothness level, then the posterior contraction rate around the truth is equal up to a log factor to the frequentist minimax rate of estimation.  
An extension to the case when the covariate space is unbounded is also established.
As our result holds without a priori knowledge of the smoothness level of the true density, the established posterior contraction rates are adaptive. 
Moreover, we show that the rate is not affected by inclusion of irrelevant covariates in the model.
In Monte Carlo simulations, a version of the model compares favorably to a cross-validated kernel conditional density estimator.

\end{abstract}

%\begin{keyword}[class=AMS]
%\kwd[Primary ]{62G07}
%\kwd[; secondary ]{62G20.}
%\end{keyword}


% JEL codes
% C11	Bayesian Analysis: General
% C14	Semiparametric and Nonparametric Methods: General

\end{frontmatter}

\section{Introduction}

Conditional distributions provide a general way to describe a relationship between a response variable and covariates.
An introduction to classical nonparametric estimation of conditional distributions and applications in economics can be found in Chapters 5-6 of \cite{LiRacine2007}.
Applications of flexible Bayesian  models for conditional densities include
analysis of financial data and distribution of earnings in \cite{Geweke:07}, estimation of health expenditures in 
\cite{KeaneStavrunova2011SMRApp}, and analysis of firms’ leverage data in \cite*{VillaniKohnNott2012}; see also  
\cite{MacEachern:99}, \cite*{DeIorioMullerRosnerMacEachern:04}, \cite{GriffinSteel:06},
\cite*{DunsonPillaiPark:07}, 
\cite{DunsonPark:08}, \cite*{VillaniKohnGiordani:07}, \cite{ChungDunson:09},
\cite*{LiVillaniKohn2010}, \cite{NoretsPelenis2012},   
and \cite{NoretsPelenis:11}.
This literature suggests that the Bayesian approach to nonparametric conditional distribution estimation has several attractive properties.
First, it does not require fixing a bandwidth or similar tuning parameters.  Instead, it provides estimates of the objects of interest where these tuning parameters are averaged out with respect to their posterior distribution. 
Second, the Bayesian approach naturally provides a measure of uncertainty  through the posterior distribution.
Third, the Bayesian approach performs well in out-of-sample prediction and Monte Carlo exercises. 
The present paper contributes to the literature on theoretical properties 
of these models and provides an explanation for their excellent performance in applications.

We focus on mixtures of Gaussian densities with covariate dependent mixing weights and a variable number of mixture components for which a prior on positive integers is specified.  Conditional on the number of mixture components, 
we model the mixing weights by a multinomial logit with a common scale parameter.
The marginal distribution of covariates is not modeled. 
This model is closely related to mixture-of-experts  
(\cite*{JacobsEtAl:91},  \cite{JordanXu:95},  \cite*{PengJacobsTanner:1996},  \cite*{WoodJiangTanner:02}), also known as smooth mixtures in econometrics (\cite{Geweke:07},  \cite{VillaniKohnGiordani:07},  \cite{Norets_aos:10}).
We study asymptotic frequentist properties of the posterior distribution in this model.

Understanding frequentist properties of Bayesian nonparametric procedures is important 
because frequentist properties, such as posterior consistency and optimal contraction rates, guarantee 
that the prior distribution is not dogmatic in a precise sense.  It is not clear how to formalize this using other approaches, especially, in high or infinite dimensional settings.
There is a considerable literature on frequentist properties of nonparametric Bayesian density estimation 
(\cite*{BarronSchervishWasserman:99}, \cite*{GhosalGhoshRamamoorthi:99},  \cite{GhosalVaart:01},  
\cite*{GhosalGhoshVaart:2000},  \cite{GhosalVandervaart:07},  \cite{Huang:04},  \cite{Scricciolo:06},  \cite{VaartZanten:09},  \cite{Rousseau:10},  \cite*{KruijerRousseauVaart:09},  \cite*{ShenTokdarGhosal2013}).
There are fewer results for conditional distribution models in which the distribution of covariates is left unspecified.  \cite{Norets_aos:10} studies approximation bounds in Kullback-Leibler distance for several classes of conditional density models.  
\cite{NoretsPelenis:11} consider posterior consistency for 
a slightly more general version of the model we consider here 
and kernel stick breaking mixtures for conditional densities.   \cite*{pati2013} study posterior consistency when mixing probabilities are modeled by transformed Gaussian processes.  \cite*{Tokdar2010}  show posterior consistency for models based on logistic Gaussian process priors. 
\cite{shen2014adaptive} obtain posterior contraction rates for a compactly supported conditional density model based on splines.

In this article, we show that under reasonable conditions on the prior, the posterior in our model
contracts at an optimal rate up to a logarithmic factor.
The assumed prior distribution does not depend on the smoothness level of the true conditional density. Thus, the obtained posterior contraction rate is adaptive across all smoothness levels.
An interpretation of this is that 
the prior 
%is not dogmatic with regard to smoothness or that it  
puts sufficient amount of weight around conditional densities of all smoothness levels and, 
thus, the posterior can concentrate around the true density of any smoothness nearly as quickly as possible.
In this particular sense, the prior is not dogmatic with regard to smoothness.

Adaptive posterior convergence rates in the context of density estimation are  obtained by 
\cite{Huang:04}, \cite{Scricciolo:06}, \cite{VaartZanten:09}, \cite{Rousseau:10}, 
\cite{KruijerRousseauVaart:09}, and \cite{ShenTokdarGhosal2013}.  
If the joint and conditional densities have the same smoothness, 
adaptive posterior contraction rates for multivariate joint densities in
\cite{VaartZanten:09} and \cite{ShenTokdarGhosal2013} imply adaptive rates for the conditional densities. 
However, it is important to note here that when the conditional density is smoother than the joint density in the sense of H\"older, it is not clear if the optimal adaptive rates for the conditional density can be achieved with a model for the joint distribution. 
A closely related concern, which is occasionally raised by researchers using mixtures for modeling a joint multivariate distribution and then extracting conditional distributions of interest,
is that many mixture components might be used primarily to provide
a good fit to the marginal density of covariates and, as a result, the fit for conditional densities deteriorates (see, for example, \cite*{wade2014improving}).
In our settings, this problem does not arise as we put a prior on the conditional density directly and do not model the marginal density of the covariates.  The resulting convergence rate depends only on the smoothness level of the conditional density.

An important advantage of estimating the conditional density directly is that the problem of covariate selection can be easily addressed.  We show that in a version of our model the posterior contraction 
rate is not affected by the presence of a fixed number of irrelevant covariates.
Also, an application of Bayesian model averaging to the covariate selection problem delivers posterior contraction rates that are not affected  by irrelevant covariates.
Thus, we can say that the posterior contraction rates we obtain are also adaptive with respect to the dimension of the relevant covariates.

Our results hold for expected total variation and Hellinger distances for conditional densities, where the expectation is taken with respect to the distribution of covariates.  
The use of these distances allows us to easily adapt a general posterior contraction theorem from 
\cite{GhosalGhoshVaart:2000} to the case of a model for conditional distribution only.  
An important part of our proof strategy is to recognize that 
our model for conditional density is consistent with a joint density that is 
a mixture of multivariate normal distributions so that we can exploit  
approximation results for mixtures of multivariate normal distributions obtained in 
\cite{deJongeVanZanten2010} and
\cite{ShenTokdarGhosal2013}.
Our entropy calculations improve considerably the bounds obtained in 
\cite{NoretsPelenis:11}.

We also evaluate the finite sample performance of our conditional density model in Monte Carlo simulations.
The model perform consistently with the established asymptotic properties and compares favorably to a cross-validated kernel conditional density estimator from
\cite{Hall2004}.

The paper is organized as follows.  
Section \ref{sec:main_results} presents the assumptions on the true conditional density, the proposed prior distributions, and the main theorem on posterior convergence rates. 
The prior thickness results are given in Section \ref{sec:Prior_thickness}. 
Section \ref{sec:sieve} describes the sieve construction and entropy calculations. 
An extension of the results to an unbounded covariate space is considered in Section \ref{sec:non_comp}. 
The presence of irrelevant covariates is analyzed in Section \ref{sec:irr_cov}.
Section \ref{sec:simul} presents results of Monte Carlo simulations.
We conclude with a discussion of the results in Section \ref{sec:conclusion}.

\section{Main results}
\label{sec:main_results}

\subsection{Notation} 
Let $\mathcal{Y} \subset \mathbb{R}^{d_y}$ be the response space, $\mathcal{X} \subset \mathbb{R}^{d_x}$ be the covariate space, and $\mathcal{Z}=\mathcal{Y} \times \mathcal{X}$. 
Let $\mathcal{F}$ denote a space of conditional densities with respect to the Lebesgue measure,
\[
\mathcal{F} = \bigg\{ f:\mathcal{Y} \times \mathcal{X} \rightarrow [0, \infty) \mbox{ - Borel measurable, } 
\int f(y|x)dy =1, \; \forall x \in \mathcal{X} \bigg\}.
\]
Suppose $(Y^n,X^n)=(Y_1,X_1,\ldots,Y_n,X_n)$ is a random sample from the joint density $f_0 g_0$, where  $f_0 \in \mathcal{F}$ and
$g_0$ is a density on $\mathcal{X}$ with respect to the Lebesgue measure. 
Let $P_0$ and $E_0$ denote the probability measure and expectation corresponding to $f_0 g_0$.
For $f_1, f_2 \in \mathcal{F}$, 
\[d_h(f_1,f_2) = \left(\int \left(\sqrt{f_1(y|x)}-\sqrt{f_2(y|x)}\right)^2 g_0(x) dy dx \right)^{1/2}\;\mbox{ and }\]  
	\[d_1(f_1,f_2) = \int  |f_1(y|x)-f_2(y|x)| g_0(x) dy dx\]
denote analogs of the Hellinger and total variation distances correspondingly.  Also, let us denote the Hellinger distance for the joint densities by $d_H$. 

Let us denote the largest integer that is strictly smaller than $\beta$ by 
$\lfloor \beta \rfloor$.
For $L:\mathcal{Z}\rightarrow[0,\infty)$, $\tau_0 \geq 0$, and $\beta>0$,
a class of locally H\"older functions, $\mathcal{C}^{\beta,L,\tau_0}$, consists of 
$f:\mathbb{R}^d\rightarrow \mathbb{R}$ such that for $k=(k_1,\ldots,k_d)$, $k_1+\cdots+k_d \leq \lfloor \beta \rfloor$, mixed partial derivative of order $k$, $D^k f$, is finite and for $k_1+\cdots+k_d = \lfloor \beta \rfloor$ and $\Delta z \in \mathcal{Z}$,
\[
|D^k f (z+\Delta z) - D^k f (z)| \leq L(z) ||\Delta z||^{\beta-\lfloor \beta \rfloor} e^{\tau_0 ||\Delta z||^2}.
\]
Operator ``$\lesssim$'' denotes less or equal up to a multiplicative positive constant relation.  $J(\epsilon, A, \rho)$ denotes the $\epsilon$-covering number of the set $A$ with respect to the metric $\rho$. 
For a finite set $A$, let $|A|$ denote the cardinality of $A$.
The set of natural numbers is denoted by  
$\mathbb{N}$. The $m$-dimensional simplex is denoted by $\Delta^{m-1}$.  $I_k$ stands for the $k \times k$ identity matrix.
Let $\phi_{\mu,\sigma}$ denote a multivariate normal density with mean $\mu \in \mathbb{R}^k$ and covariance 
matrix $\sigma^2 I_k$ (or a diagonal matrix with squared elements of $\sigma$ on the diagonal,
when $\sigma$ is a $k$-vector).

\subsection{Assumptions about data generating process}
\label{sec:asns}

First, we assume that $f_0 \in \mathcal{C}^{\beta,L,\tau_0}$.
Second, we assume that $\mathcal{X}=[0,1]^{d_x}$, except for Section \ref{sec:non_comp} where we consider possibly unbounded $\mathcal{X}$.  
Third, $g_0$ is assumed to be bounded above. 
Fourth,  for all $k \leq \lfloor \beta \rfloor$ and some $\varepsilon>0$,
	\begin{equation}
	\label{eq:asnE0Dff0_Lf0}
	\int_{\mathcal{Z}} \left|\frac{D^k f_0(y|x)}{f_0(y|x)}\right|^{(2\beta+\varepsilon)/k} f_0(y|x) dydx< \infty,
	\;
	\int_{\mathcal{Z}} \left|\frac{L(y,x)}{f_0(y|x)}\right|^{(2\beta+\varepsilon)/\beta} f_0(y|x) dydx< \infty.
%	E_0 \left([L/f_0]^{(2\beta+\epsilon)/\beta} \right) < \infty;
	\end{equation} 
	%\begin{equation}
	%\label{eq:E0Lf0}
%E_0 \left([L/f_0]^{(2\beta+\epsilon)/\beta} \right) < \infty;
	%\end{equation} 
Finally, for all $x \in \mathcal{X}$, all sufficiently large $y \in \mathcal{Y}$ and some positive $(c,b,\tau)$, 
\begin{equation}
	\label{eq:asnf0_exp_tails}
f_0(y|x) \leq c \exp(-b ||y||^\tau). 
	\end{equation}

\subsection{Prior}
\label{sec:prior}

The prior, $\Pi$, on $\mathcal{F}$ is defined by a location mixture of normal densities 
\begin{equation}
\label{eq:cond_mix_def}
p(y|x, \theta, m) = \sum_{j=1}^m\frac{ \alpha_j \exp\{-0.5||x-\mu_j^x||^2/\sigma^2 \}   
}
{\sum_{i=1}^m \alpha_i \exp\{-0.5||x-\mu_i^x||^2/\sigma^2 \}   
}\phi_{\mu_j^y,\sigma}(y),
\end{equation}
and a prior on $m \in \mathbb{N}$ and $\theta=(\mu_j^y, \mu_j^x, \alpha_j, j=1,2,\ldots; \sigma)$,
where 
$\mu_j^y \in \mathbb{R}^{d_y}$, $\mu_j^x \in \mathbb{R}^{d_x}$, $\alpha_j \in [0,1]$, $\sigma \in (0,\infty)$. 
The covariate dependent mixing weights are modeled by multinomial logit with restrictions on the coefficients and a common scale parameter $\sigma$.  %To simplify notations, let $p_j(x) =  \alpha_j K_j(x) / \sum_{l=1}^m \alpha_l K_l(x)$ where $K_j(x) = \exp\{-0.5||x-\mu_j^x||^2/\sigma^2 \}$.  
To facilitate simpler notations and shorter proofs, we assume $\sigma$ 
to be the same for all components of $(y,x)$, except for Section \ref{sec:irr_cov}. %\footnote{  
Extensions to component-specific $\sigma$'s, which would result in near optimal posterior contraction rates for anisotropic $f_0$, can be done along the lines of Section 5 in \cite{ShenTokdarGhosal2013}.
%} 

We assume the following conditions on the prior.  For positive constants $a_1, a_2,  \ldots, a_{9}$, the prior for $\sigma$ satisfies 
\begin{eqnarray}  
	\Pi( \sigma^{-2} \geq s)  &\leq& a_1 \exp \{-a_2 s^{a_3} \} \quad \text{for all sufficiently large} \, \,  s > 0 
	\label{eq:asnPrior_sigma1}\\
	\Pi( \sigma^{-2} < s)  &\leq& a_4 s ^{a_5}  \quad \text{for all sufficiently small}  \, \,  s > 0 \label{eq:asnPrior_sigma2}\\
	\Pi\{ s < \sigma^{-2} < s(1+t) \} &\geq& a_6 s^{a_7} t^{a_8} \exp \{-a_9 s^{1/2}\}, \quad s > 0, \quad t \in (0,1). 
	\label{eq:asnPrior_sigma3}
\end{eqnarray}
An example of a prior that satisfies \eqref{eq:asnPrior_sigma1}-\eqref{eq:asnPrior_sigma2} is the inverse Gamma prior 
for $\sigma$.  The usual conditionally conjugate inverse Gamma
prior for $\sigma^2$ satisfies  \eqref{eq:asnPrior_sigma1} and  \eqref{eq:asnPrior_sigma2}, but not 
 \eqref{eq:asnPrior_sigma3}.  \eqref{eq:asnPrior_sigma3} requires the probability to values of $\sigma$ near $0$ to be 
 higher than the corresponding probability for inverse Gamma prior for $\sigma^2$.  
This assumption is in line with the previous work on adaptive posterior contraction rates for mixture models, see 
\cite{KruijerRousseauVaart:09,shen2014adaptive}.
Prior for 
$(\alpha_1,\ldots,\alpha_m)$ given $m$ is Dirichlet$(a/m,\ldots,a/m)$, $a > 0$.
\begin{equation}
	\label{eq:asnPrior_m}
	\Pi(m=i) \propto \exp(-a_{10} i (\log i)^{\tau_1}), i =2, 3, \ldots, \quad a_{10}>0, \tau_1 \geq 0.
\end{equation}
A priori, $\mu_{j}=(\mu_{j}^y,\mu_{j}^x)$'s are independent from other parameters and across $j$,
and $\mu_{j}^y$ is independent of  $\mu_{j}^x$.
Prior density for $\mu_{j}^x$ is bounded away from 0 on $\mathcal{X}$ and equal to 0 elsewhere.	
Prior density for $\mu_{j}^y$ is bounded below for some $a_{12}, \tau_2 > 0$ by 
\begin{equation}
	\label{eq:asnPrior_mu_lb}
	a_{11}\exp(-a_{12} ||\mu_j^y||^{\tau_2} ),
\end{equation}
	and for some $a_{13}, \tau_3>0$ and all sufficiently large $r > 0$,
	\begin{equation}
	\label{eq:asnPrior_mu_tail_ub}
	1- \Pi(\mu_{j}^y \in [-r,r]^{d_y}) \leq \exp(-a_{13} r^{\tau_3}). 
\end{equation}

\subsection{Results}  
To prove the main result, we adapt a general posterior contraction theorem to the case of conditional densities.
We define the Hellinger, total variation, and Kullback-Leibler distances for conditional distributions as special cases of the corresponding distances for the joint densities.  Therefore, the proof of the following result is essentially the same as the proof of Theorem 2.1 in \cite{GhosalVaart:01} and is omitted here. 

\begin{theorem} 
\label{th:gen_contr_rate}
 Let $\epsilon_n>0$ be a sequence such that $n \epsilon_n^2 \to \infty$. Let $\rho$ be $d_h$ or $d_1$.
Suppose $\mathcal{F}_n \subset \mathcal{F}$ is a sieve with the following bound on the metric entropy $J( \epsilon_n, \mathcal{F}_n, \rho)$
	\begin{eqnarray} \label{eq:entropy}
	\log J( \epsilon_n, \mathcal{F}_n, \rho) \leq c_1 n \epsilon_n^2, 
	\end{eqnarray}
%		and for $\tilde{\epsilon}_n \leq \epsilon_n$,
	\begin{eqnarray}\label{eq:sievecomplement}
	\Pi(\mathcal{F}_n^c) \leq c_3 \exp\{ -(c_2+4)n \tilde{\epsilon}_n^2\}, \; \tilde{\epsilon}_n \leq \epsilon_n,
	\end{eqnarray}
	and for a generalized Kullback-Leibler neighborhood 
	\[\mathcal{K}(f_0,\epsilon)=\left \{f: \; \int f_0 g_0 \log(f_0/f)< \epsilon^2, \; \int f_0 g_0 [\log(f_0/f)]^2 < \epsilon^2 \right \},\]
	\begin{equation}
	\label{eq:prior_thick}
	\Pi(\mathcal{K}(f_0,\tilde{\epsilon}_n)) \geq c_4 \exp\{ -c_2 n \tilde{\epsilon}_n^2\}.
	\end{equation}

Then, there exists $M>0$ such that 
\[
\Pi\left( f: \rho(f,f_0) > M \epsilon_n | Y_1,X_1,\ldots, Y_n,X_n\right) \stackrel{P_0^n}{\rightarrow} 0.
\]

\end{theorem}

The theorem's assumptions \eqref{eq:entropy}-\eqref{eq:prior_thick} have the following interpretation.
The prior probability outside of an increasing sequence of compact subsets of the parameter space 
is required to be appropriately small and
the prior needs to put sufficient probability on the Kullback-Leibler neighborhoods of the true density.
The theorem essentially describes a trade-off between these two requirements.
The theorem is formulated with $\tilde{\epsilon}_n$ and $\epsilon_n$ to provide more insight on how 
bounds on the sieve's entropy and bounds on the prior probability of the Kullback-Leibler neighborhoods of $f_0$ and the sieve's complement
differently restrict the contraction rate.
%For a more detailed discussion of the theorem, we refer the reader to \cite{GhosalVaart2013theory}.

\begin{theorem} 
\label{th:rate_cond_dens}
Under the assumptions in Sections \ref{sec:asns}-\ref{sec:prior},  
the sufficient conditions of Theorem \ref{th:gen_contr_rate} hold with 
\[
\epsilon_n=n^{-\beta/(2\beta+d)} (\log n)^t,
\] 
where
$t > t_0 +  \max \{0, (1- \tau_1)/2\}$, 
$t_0=  (ds + \max\{\tau_1,1,\tau_2/\tau\}) / (2 + d/\beta)$, 
$d=d_y+d_x$,
and  $s = 1 + 1/\beta + 1/\tau$. 
\end{theorem}

The proof of the theorem is divided into two main parts.  First,
we establish the prior thickness condition \eqref{eq:prior_thick}
in Theorem \ref{th:prior_thickness}.
Then, the conditions on the sieve are established in Theorems \ref{th:sieve} and \ref{tm:sieve_n}.

\section{Prior thickness}
\label{sec:Prior_thickness}

The prior thickness condition is formally proved in Theorem \ref{th:prior_thickness}.
Let us briefly describe the main steps of the proof 
placing it in the context of the previous literature.
First, we recognize that
the covariate dependent mixture defined in \eqref{eq:cond_mix_def} is consistent with the following 
mixture of normals for the joint distribution of $(y,x)$,
\begin{equation}
\label{eq:joint_mix_def}
p(y,x| \theta, m) = \sum_{j=1}^m \alpha_j \phi_{\mu_j,\sigma}(y,x),
\end{equation}
where $\mu_j = (\mu_j^y,\mu_j^x)$.

Second, we bound the Hellinger distance between conditional densities $f_0(y|x)$ and 
$p(y|x,\theta, m)$ by a distance between the joint densities 
$f_0(y|x)u(x)$ and 
$p(y,x|\theta, m)$, where $u(x)$ is a uniform density on 
$\mathcal{X}$.  It is important to note that $f_0(y|x) u(x)$ has the same smoothness level as $f_0(y|x)$.

Third, we obtain a suitable approximation for the joint distribution
$f_0(y|x) u(x)$ by mixtures $p(y,x|\theta, m)$ using modified results  
from \cite{ShenTokdarGhosal2013}.
The idea of the approximation argument is introduced in 
\cite{Rousseau:10} in the context of approximation 
of a univariate density by mixtures of beta densities.
\cite{KruijerRousseauVaart:09} use this idea for obtaining approximation results for mixtures of univariate normal densities.
\cite{deJongeVanZanten2010} extend the idea to approximation of multivariate functions, 
but the functions they approximate are not necessarily densities and their weights $\alpha_j$'s could be negative.  \cite{ShenTokdarGhosal2013} use the same techniques with an additional step
to approximate multivariate densities by mixtures with 
$\alpha_j$'s belonging to a simplex.  
It is not clear whether the mixing weights they obtain are actually non-negative.
In Lemma \ref{lm:positive_mix_dens} in the appendix, we state a modified version of their 
Theorem 3 that ensures non-negativity of the weights.
With a suitable approximation at hand, verification 
of condition \eqref{eq:prior_thick} proceeds along the lines of similar results in 
\cite{GhosalVaart:01}, \cite{GhosalVandervaart:07},
\cite{KruijerRousseauVaart:09}, and, especially, \cite{ShenTokdarGhosal2013}, with modifications necessary to handle the case of conditional distributions.

\begin{theorem} \label{th:prior_thickness}
Suppose the assumptions from Sections \ref{sec:asns}-\ref{sec:prior} hold.
Then, for any $C > 0$ and all sufficiently large $n$,
\begin{eqnarray}\label{eq:KL}
\Pi( \mathcal{K}(f_0,  \tilde{\epsilon}_n)) \geq \exp \{-C n \tilde{\epsilon}_n^2 \},
\end{eqnarray}  
where $\tilde{\epsilon}_n =  n^{-\beta/(2\beta + d)} (\log n)^t$, 
$t > (ds + \max\{\tau_1,1,\tau_2/\tau\}) / (2 + d/\beta)$, 
$s = 1 + 1/\beta + 1/\tau$, and
$(\tau,\tau_1, \tau_2)$ are defined in Sections \ref{sec:asns} and \ref{sec:prior}.
\end{theorem}

\begin{proof}

By Lemma \ref{lm:dH_cond_bdd_dHjoint}, for $p(\cdot | \cdot, \theta, m)$ defined in \eqref{eq:joint_mix_def},
\begin{eqnarray}
d_h^2 (f_0, p(\cdot | \cdot, \theta, m)) &=&  \int ( \sqrt{f_0(y | x)} -  \sqrt{p(y | x, \theta, m)})^2 g_0(x) dydx  \notag \\
&\leq& C_1 \int  ( \sqrt{f_0(y | x)u(x)} -  \sqrt{p(y, x | \theta, m)})^2  d(y, x)  \notag \\
\label{eq:cond_bd_joint}
&=& C_1 d_H^2(f_0 u, p(\cdot |\theta, m) ),
\end{eqnarray}
where $u(x)$ is a uniform density on $\mathcal{X}$.  

For 
$\sigma_n = [\tilde{\epsilon}_n / \log (1/ \tilde{\epsilon}_n) ]^{1/\beta}$,
$\varepsilon$ defined in \eqref{eq:asnE0Dff0_Lf0},
a sufficiently small $\delta>0$,
$b$ and $\tau$ defined in \eqref{eq:asnf0_exp_tails},
$a_0 = \{ (8\beta + 4\varepsilon +16)/(b \delta)\}^{1/\tau}$,
$a_{\sigma_n} = a_0 \{\log (1/\sigma_n) \}^{1/\tau}$, 
and $b_1 > \max \{1, 1/ 2\beta \}$ satisfying $\tilde{\epsilon}_n^{b_1}  \{  \log (1/ \tilde{\epsilon}_n) \}^{5/4} \leq \tilde{\epsilon}_n$,
the proof of Theorem 4 in \cite{ShenTokdarGhosal2013} implies the following three claims.
First,
there exists a partition of $\{z \in \mathcal{Z}: ||z|| \leq a_{\sigma_n}\}$,
$\{U_j, j=1,\ldots,K\}$ such that for $j=1,\ldots,N$,
$U_j$ is a ball with diameter $\sigma_n \tilde{\epsilon}_n^{2 b_1}$
and center $z_j = (x_j, y_j)$;
for $j=N+1,\ldots,K$,  
$U_j$ is a set with a diameter bounded above by $\sigma_n$;
$1 \leq N < K \leq C_2 \sigma_n^{-d} \{\log (1/ \tilde{\epsilon}_n) \}^{d +d/\tau}$, 
where $C_2>0$ does not depend on $n$.
Second, there exist 
$\theta^\star = \{\mu_j^\star, \alpha_j^\star, j = 1,2,\ldots; \sigma_n\}$ with
$\alpha_j^\star=0$ for $j > N$, $\mu_j^\star=z_j$ for $j=1,\ldots,N$, and
$\mu_j^\star \in U_j$ for $j=N+1,\ldots,K$
 such that 
for $m=K$ and a positive constant $C_3$,
\begin{equation}
\label{eq:f0upsbeta}
d_H(f_0 u, p(\cdot |\theta^\star, m) ) \leq C_3 \sigma_n^\beta.
\end{equation}
Third, there exists constant $B_0>0$ such that
\begin{equation}
\label{eq:P0_zgeqa}
P_0(\norm{z} > a_{\sigma_n}) \leq B_0 \sigma_n^{4\beta + 2\varepsilon +8}.
\end{equation}

For $\theta$ in set
\begin{align*}
	S_{\theta^\star}=&\big \{
	(\mu_j, \alpha_j, \, j=1,2,\ldots; \sigma): 
	 \; \mu_j  \in U_j, j=1, \ldots, K, \\
	& \sum_{j=1}^K \abs{\alpha_j - \alpha_j^\star} \leq 2\tilde{\epsilon}_n^{2db_1}, \, \min_{j=1, \ldots, K} \alpha_j \geq \tilde{\epsilon}_n^{4db_1}/2, 
	%\\	& 
	\sigma^2 \in [ \sigma_n^{2}/(1+ \sigma_n^{2\beta}), \sigma_n^{2} ]
\big \},
\end{align*}
 we have 
\begin{align*}
& d_H^2(p(\cdot | \theta^\star, m),  p(\cdot | \theta, m))  \leq 
\norm{\sum_{j=1}^K \alpha_j^\star \phi_{\mu_j^\star,\sigma_n} - \sum_{j=1}^K \alpha_j \phi_{\mu_j,\sigma}}_1\\
&\leq  \sum_{j=1}^K \abs{\alpha_j^\star-\alpha_j} 
+ \sum_{j=1}^N \alpha_j^\star 
\left[\norm{\phi_{\mu_j^\star,\sigma_n} - \phi_{\mu_j,\sigma_n}}_1
+
\norm{\phi_{\mu_j,\sigma_n} - \phi_{\mu_j,\sigma}}_1 \right].
%\\
%&\leq&   2\tilde{\epsilon}_n^{2d b_1} + \sum_{j=1}^N p_j \frac{\norm{\mu_j - z_j}}{\sigma}  \leq   2\tilde{\epsilon}_n^{2b_1}. 
\end{align*}
For $j=1,\ldots,N$,
$
\norm{\phi_{\mu_j^\star,\sigma_n} - \phi_{\mu_j,\sigma_n}}_1 \leq ||\mu_j^\star - \mu_j||/\sigma_n \leq 
\tilde{\epsilon}_n^{2b_1}
$.  Also, 
\begin{equation}
\label{eq:bd4sigmas}
\norm{\phi_{\mu_j, \sigma_n} - \phi_{\mu_j, \sigma}}_1  
\leq \sqrt{d/2} 
\abs{\frac{\sigma_n^2}{\sigma^2} - 1 - \log \frac{\sigma_n^2}{\sigma^2}}^{1/2} \leq  C_4 \sqrt{d/2} \abs{\frac{\sigma_n^2}{\sigma^2} - 1}
\lesssim  \sigma_n^{2\beta}, 
\end{equation}
where the penultimate inequality follows from the fact that 
$\abs{\log x - x +1} \leq C_4\abs{x-1}^2$ for $x$ in a neighborhood of 1 and some $C_4 > 0$.   Hence,   
 $d_H (p(\cdot | \theta, m), p(\cdot | \theta^\star, m)) \lesssim \sigma_n^{\beta}$  and, 
by \eqref{eq:cond_bd_joint}, \eqref{eq:f0upsbeta} and the triangle inequality, 
$d_h (f_0, p(\cdot | \cdot, \theta, m)) \leq C_5 \sigma_n^\beta$
for some $C_5>0$, all $\theta \in S_{\theta^\star}$, and $m=K$.

Next, for $\theta \in S_{\theta^\star}$, let us consider a lower bound on the ratio 
$p(y | x, \theta, m)/f_0(y | x)$.  
Note that $\sup_{y,x} f_0(y | x)<\infty$ and
$p(y | x, \theta, m) \geq \sigma^{d_x} p(y, x | \theta, m)$.
For $z \in \mathcal{Z}$ with $\norm{z} \leq a_{\sigma_n}$, there exists $J \leq K$ for which 
$||z-\mu_J || \leq \sigma_n$. Thus, 
for all sufficiently large $n$ such that $\sigma_n^2/\sigma^2 \leq 2$,
$p(z | \theta, m) \geq \min_j \alpha_j \cdot \phi_{\mu_J,\sigma} (z) \geq
[\tilde{\epsilon}_n^{4db_1}/2] \cdot \sigma_n^{-d} e^{-1} / (2\pi)^{d/2}$ and 
\begin{equation}
\label{eq:lambda_def}
\frac{p(y | x, \theta, m)}{f_0(y | x)} \geq
C_6 \tilde{\epsilon}_n^{4db_1} \sigma_n^{-d_y}, \mbox{ for some } C_6>0.
\end{equation}
For $z \in \mathcal{Z}$ with $\norm{z} > a_{\sigma_n}$, 
$\norm{ z - \mu_j}^2 \leq 2 (\norm{z}^2 + \norm{\mu}^2) \leq 4\norm{z}^2$
for all $j=1,\ldots,K$. Thus, for all sufficiently large $n$,
$
p(z | \theta, m) \geq \sigma_n^{-d} \exp (-4\norm{z}^2/\sigma_n^2) / (2\pi)^{d/2}
$ and 
\begin{equation*}
\frac{p(y | x, \theta, m)}{f_0(y | x)} \geq
C_7 \sigma_n^{-d_y} \exp (-4\norm{z}^2/\sigma_n^2), \mbox{ for some } C_7>0.
\end{equation*}

Denote the lower bound in \eqref{eq:lambda_def} by $\lambda_n$ and consider all sufficiently large $n$ such that 
$\lambda_n < e^{-1}$.
For any $\theta \in S_{\theta^\star}$, 
\begin{align*}
& \int \bigg( \log \frac{f_0(y|x)}{p(y|x, \theta, m)}\bigg)^2 
1\left \{ \frac{p(y|x, \theta, m)}{f_0(y|x)} < \lambda_n \right\}
f_0(y|x) g_0(x) dydx \\
& =
\int \bigg( \log \frac{f_0(y|x)}{p(y|x, \theta, m)}\bigg)^2 1\left \{ \frac{p(y|x, \theta, m)}{f_0(y|x)} < \lambda_n, || (y,x)|| > a_{\sigma_n} \right\} f_0(y|x) g_0(x) dydx 
\\
& \leq
\frac{4}{\sigma_n^4} \int_{\norm{z} > a_{\sigma_n}} 
\norm{z}^4 f_0 g_0  dz \leq \frac{4}{\sigma_n^4} E_0(\norm{Z}^8)^{1/2} (P_0(\norm{Z} > a_{\sigma_n}))^{1/2}  \leq C_8 \sigma_n^{2\beta + \varepsilon}
 \end{align*}
 for some constant $C_8$.  The last inequality follows from 
\eqref{eq:P0_zgeqa} and tail condition in \eqref{eq:asnf0_exp_tails}.
Also note that 
 \begin{eqnarray*}
 \log \frac{f_0(y|x)}{p(y|x, \theta, m)} 1\left\{\frac{p(y|x, \theta, m)}{f_0(y|x)} < \lambda_n \right\} \leq \left\{ \log 
\frac{f_0(y|x)}{p(y|x, \theta, m)}\right\}^2 1\left\{\frac{p(y|x, \theta, m)}{f_0(y|x)} < \lambda_n \right\}
 \end{eqnarray*}
 and, thus, 
 \begin{eqnarray*}
 \int \log \frac{f_0(y|x)}{p(y|x, \theta, m)} 1 \left \{ \frac{p(y|x, \theta, m)}{f_0(y|x)} < \lambda_n \right \} f_0 g_0 dz  \leq C_8\sigma_n^{2\beta + \varepsilon}.
 \end{eqnarray*}
By Lemma \ref{lm:dH_KL},  both
$E_0 (\log (f_0(Y|X)/p(Y|X, \theta, m)))$ and $E_0 ([\log (f_0(Y|X)/p(Y|X, \theta, m))]^2)$ are bounded by 
$C_9 \log (1/\lambda_n)^2\sigma_n^{2\beta} \leq A\tilde{\epsilon}_n^2$ for some constant $A$.  

%note that for some $C_8>0$, $\log (1/ \lambda_n) \leq C_8 \log ( 1/ \tilde{\epsilon}_n)$

Finally, we calculate a lower bound on the prior probability of $m = K$ and $\{\theta \in S_{\theta^\star}\}$.
By \eqref{eq:asnPrior_m}, for some $C_{10}>0$,
\begin{equation}
\Pi(m=K) \propto \exp[-a_{10} K (\log K)^{\tau_1}] \geq \exp [-C_{10}\tilde{\epsilon}_n^{-d/\beta} \{\log (1/ \tilde{\epsilon}_n)\}^{d + d/\beta+ d/\tau + \tau_1 }] \label{eq:KL1}. 
\end{equation}
From Lemma 10 of \cite{GhosalVandervaart:07}, for some constants $C_{11}, C_{12}>0$ and all sufficiently large $n$,
\begin{align}
&\Pi\left( \sum_{j=1}^K \abs{\alpha_j - \alpha_j^\star} \geq 2 \tilde{\epsilon}_n ^{2db_1}, \min_{j=1, \ldots, K} \alpha_j \geq \tilde{\epsilon}_n^{4db_1}/2 \bigg| m=K\right) \geq \exp[ -C_{11} K \log (1/\tilde{\epsilon}_n)]  
\notag \\
&\geq 
\exp [-C_{12} \tilde{\epsilon}_n^{-d/\beta} \{\log (1/ \tilde{\epsilon}_n)\}^{d/\beta + d/\tau +d +1}].
\end{align}
For $\pi_{\mu}$ denoting the prior density of $\mu_j^y$ and some $C_{13}, C_{14}>0$, \eqref{eq:asnPrior_mu_lb} implies
\begin{align}
&\Pi(\mu_j \in U_j, j=1, \ldots, N ) \geq \{C_{13}\pi_{\mu}(a_{\sigma}) \mbox{diam}(U_1)^d \}^N
\notag \\ 
&\geq 
\exp \left [ -C_{14} \tilde{\epsilon}_n^{-d/\beta}
\left \{\log (1/ \tilde{\epsilon}_n)\right\}^{d + d/\beta + d/\tau +\max\{1, \tau_2/\tau\}}\right]
\label{eq:Pr_muU}
\end{align}
Assumption \eqref{eq:asnPrior_sigma3} on the prior for $\sigma$, implies
%For the assumed inverse gamma prior for $\sigma$, the mean value theorem implies
\begin{equation}
\Pi(\sigma^{-2} \in \{ \sigma_n^{-2}, \sigma_n^{-2}(1+ \sigma_n^{2\beta}) \}) \geq a_{8} \sigma_n^{-2 a_7} 
 \sigma_n^{2\beta a_8} \exp\{ - a_9\sigma_n^{-1} \} \geq
\exp\{ - C_{15}\sigma_n^{-1} \}\label{eq:KL4}. 
\end{equation} 
It follows from \eqref{eq:KL1} - \eqref{eq:KL4},  
that for all sufficiently large $n$,  $s = 1 + 1/\beta + 1/\tau$, and some $C_{16}>0$ 
\begin{eqnarray*}
\Pi( \mathcal{K}(f_0, A \tilde{\epsilon}_n))  \geq \Pi(m=N, \theta_p \in S_{\theta_p} )  \geq 
\exp [-C_{16}\tilde{\epsilon}_n^{-d/\beta} \{\log (1/ \tilde{\epsilon}_n)\}^{ds + \max\{\tau_1,1,\tau_2/\tau\} }] .
\end{eqnarray*}
The last expression of the above display is bounded below by $\exp\{-C n \tilde{\epsilon}_n^2 \}$ for any $C>0$,
$\tilde{\epsilon}_n =  n^{-\beta/(2\beta + d)} (\log n)^t$, any $t > (ds + \max\{\tau_1,1,\tau_2/\tau\}) / (2 + d/\beta)$, and all sufficiently large $n$.  Since the inequality in the definition of $t$ is strict, the claim of the theorem follows immediately.

\end{proof}

\section{Sieve construction and entropy bounds} \label{sec:sieve}
For $H \in  \mathbb{N}$, $0<\underline{\sigma}<\overline{\sigma}$, and $\overline{\mu},\underline{\alpha}>0$,
let us define a sieve
\begin{eqnarray}\label{eq:sieve}
\mathcal{F} = \{p(y|x,\theta,m): \; m\leq H, \; \alpha_j \geq \underline{\alpha}, \; 
\sigma \in [\underline{\sigma}, \overline{\sigma}], \mu_j^y \in [-\overline{\mu}, \overline{\mu}]^{d_y}, j=1, \ldots,m\}.
\end{eqnarray}
In the following theorem, we bound the covering number of $\mathcal{F}$
in norm 
\[d_{SS}(f_1, f_2) = \sup_{x \in \mathcal{X}} \norm{f_1(y|x) - f_2(y|x)}_1.\]

\begin{theorem} 
\label{th:sieve}
%Analog of Proposition 2 from \cite{ShenTokdarGhosal2013}. 
For $0<\epsilon < 1$ and $\underline{\sigma}\leq 1$, 
\begin{align*}
J(\epsilon, \mathcal{F}, d_{SS}) \leq & 
H \cdot \left \lceil \frac{16\overline{\mu}d_y}{\underline{\sigma}\epsilon} \right \rceil^{Hd_y}
\cdot \left \lceil \frac{48 d_x}{\underline{\sigma}^2 \epsilon} \right \rceil^{Hd_x}
\cdot H \left \lceil \frac{ \log (\underline{\alpha}^{-1}) }{\log (1 + \epsilon/[12H])} \right \rceil^{H-1}
\\ 
& \cdot \left \lceil \frac{ \log (\overline{\sigma}/\underline{\sigma})} {\log (1 + \underline{\sigma}^2\epsilon/[48\max\{d_x,d_y\}])} \right \rceil.
\end{align*} 
For $\underline{\alpha} \leq 1/2$, all sufficiently large $H$, 
large $\overline{\sigma}$ and small $\underline{\sigma}$, 
\begin{align*}
\Pi(\mathcal{F}^c)  \leq 
&
H^2 \exp\{-a_{13} \overline{\mu} ^{\tau_3}\}  + H^2 \underline{\alpha}^{a/H} +   \exp\{- a_{10} H(\log H)^{\tau_1}  \} 
%+ \exp\{-b_4 \log \overline{\sigma}\} + \exp\{-b_5 \underline{\sigma}^{-2a_3}\}
\\
&
+ a_1\exp \{-a_2 \underline{\sigma}^{-2 a_3}\} 
+ 
a_4 \exp\{-2a_5 \log \overline{\sigma}\}
.  
\end{align*}  
\end{theorem}

\begin{proof}
We will start with the first assertion.  
Fix a value of $m$. 
Define set $S_{\mu^y}^m$ to
contain centers of $|S_{\mu^y}^m| =  \lceil 16\overline{\mu}d_y/(\underline{\sigma}\epsilon) \rceil$ equal length intervals 
partitioning $[-\overline{\mu}, \overline{\mu}]$. 
Similarly, define set $S_{\mu^x}^m$
to contain centers of $|S_{\mu^x}^m| =  \lceil 48d_x/(\underline{\sigma}^2 \epsilon) \rceil$ equal length intervals 
partitioning $[0, 1]$.

For $N_{\alpha} = \lceil \log (\underline{\alpha}^{-1}) /\log (1 + \epsilon/(12m))  \rceil$, define 
\[Q_\alpha=\{
\gamma_j, \, j=1,\ldots,N_\alpha: \; \gamma_1=\underline{\alpha}, \; (\gamma_{j+1} - \gamma_j) / \gamma_j = \epsilon/(12m), \,
j=1,\ldots,N_\alpha-1
\}\]
and note that for any $\gamma \in [\underline{\alpha},1]$ there exists $j\leq N_\alpha$ such that 
$0 \leq (\gamma-\gamma_j)/\gamma_j \leq \epsilon/(12m)$.
Let 
$S_\alpha^m=\{
(\tilde{\alpha}_1,\ldots,\tilde{\alpha}_m) \in \Delta^{m-1}: \; \tilde{\alpha}_{j_k} \in Q_\alpha, \, 
1 \leq j_1<j_2<\ldots< j_{m-1}\leq m\}$. Note that $|S_\alpha^m| \leq m (N_\alpha)^{m-1}$.
Let us consider an arbitrary $\alpha \in \Delta^{m-1}$.
Since $S_\alpha^m$ is permutation invariant, we can assume without loss of generality that $\alpha_m \geq 1/m$.
By definition of $S_\alpha^m$, there exists $\tilde{\alpha} \in S_\alpha^m$ such that
$0 \leq (\alpha_j-\tilde{\alpha_j})/\tilde{\alpha_j} \leq \epsilon/(12m)$ for $j=1,\ldots,m-1$.
Also, 
\[
\frac{|\alpha_m-\tilde{\alpha}_m|}{\min(\alpha_m, \tilde{\alpha}_m)}
=\frac{|\alpha_m-\tilde{\alpha}_m|}{\alpha_m}
= 
\frac{\sum_{j=1}^{m-1}\tilde{\alpha}_j(\alpha_j-\tilde{\alpha}_j)/\tilde{\alpha}_j}{\alpha_m}\leq 
\frac{\epsilon}{12}.
\]

%We consider two sets of grid points for $\sigma$ and take their union.  Let $S_{1\sigma}$ be an $\epsilon \underline{\sigma}/24d_x$-net of 
%$[1/\overline{\sigma}, 1/\underline{\sigma}]$ and $S_{2\sigma} =  \{\sigma^{l},l=1, \ldots, N_{\sigma} = \lfloor{ \log (\overline{\sigma}/\underline{\sigma})/ (\log (1 + \epsilon/4d_y))}, \sigma^1= \underline{\sigma}, (\sigma^{l+1} - \sigma^l)/\sigma^l = \epsilon/4d_y\}$.  Define $S_{\sigma} = S_{1\sigma} \cup S_{2\sigma}$. 

Define $S_{\sigma} =  \{\sigma^{l},l=1, \ldots, N_{\sigma} = \lceil{ \log (\overline{\sigma}/\underline{\sigma})/ (\log (1 + \underline{\sigma}^2\epsilon/(48\max\{d_x,d_y\})}\rceil, \sigma^1= \underline{\sigma}, (\sigma^{l+1} - \sigma^l)/\sigma^l = \underline{\sigma}^2\epsilon/(48\max\{d_x,d_y\})\}$.   Then $\abs{S_{\sigma}} = N_{\sigma}.$
%Define $S_{\sigma} = S_{1\sigma} \cup S_{2\sigma}$. 

Below we show that 
\[
S_{\mathcal{F}} =  \{p(y|x,\theta,m): \; m\leq H, 
\, 
\alpha \in S_{\alpha}^m, \,
\sigma \in S_{\sigma}, \,
\mu_{jl}^x \in S_{\mu^x}^m,\, 
\mu_{jk}^y \in S_{\mu^y}^m, \, j\leq m, \, l \leq d_x, \, k\leq d_y
 \}\] provides an $\epsilon$-net for $\mathcal{F}$ in $d_{SS}$.  Fix $p(y|x,\theta,m)  \in \mathcal{F}$ for some $m \leq H, \alpha \in \Delta^{m-1}$ with $\alpha_j \geq \underline{\alpha}$, $\mu^x \in [0, 1]^{d_x}$, $\mu^y \in [-\overline{\mu}, \overline{\mu}]^{d_y}$ and $\sigma \in [\underline{\sigma}, \overline{\sigma}]$ with $\sigma^l \leq \sigma \leq \sigma^{l+1}$.  
Find   
$\tilde{\alpha} \in S_{\alpha}^m$,
$\tilde{\mu}_{jl}^x \in S_{\mu^x}^m$,
$\tilde{\mu}_{jk}^y \in S_{\mu^y}^m$, and $\tilde{\sigma}=\sigma_l \in S_{\sigma}$
such that for all
$j=1, \ldots, m$, $k=1,\ldots, d_y$, and $l=1,\ldots, d_x$
\begin{equation*}
|\mu_{jk}^y  - \tilde{\mu}_{jk}^y | \leq \frac{\underline{\sigma} \epsilon}{16 d_y},  
\;
|\mu_{jl}^x- \tilde{\mu}_{jl}^x | \leq \frac{\underline{\sigma}^2 \epsilon}{96 d_x},  
\; 
\frac{\alpha_j - \tilde{\alpha}_j}{\alpha_j}  \leq \frac{\epsilon}{12},
\; 
\frac{|\sigma - \tilde{\sigma}|}{\sigma} \leq \frac{\underline{\sigma}^2\epsilon}{48\max\{d_x,d_y\}}. 
\end{equation*}

Let $K_j=\exp\{-0.5||x-\mu_j^x||^2/\sigma^2 \}$.
The proof of Proposition 3.1 in \cite{NoretsPelenis:11} implies the following inequality for any $x \in \mathcal{X}$
\footnote{
\cite{NoretsPelenis:11} use this inequality in conjunction with a lower bound on $K_j$, which leads to entropy bounds that are not sufficiently tight for adaptive contraction rates.
}
\begin{align*}
%\label{eq:L1normBd}
& \int |p(y|x, \theta, m) - p(y|x, \tilde{\theta}, m)| dy \leq 
2 \max_{j=1,\ldots,m}
|| \phi_{\mu_j^y,\sigma} -\phi_{\tilde{\mu}_j^y,\tilde{\sigma}}||_1 \\
%\label{eq:DiffPiBd}
& +
2 \left( \frac{\sum_{j=1}^m \alpha_j | K_j - \tilde{K}_j|}{\sum_{j=1}^m \alpha_j K_j}  +   \frac{\sum_{j=1}^m \tilde{K}_j | \alpha_j - \tilde{\alpha}_j|}{\sum_{j=1}^m \alpha_j K_j} \right).
\end{align*}
It is easy to see that
\[
|| \phi_{\mu_j^y,\sigma} -\phi_{\tilde{\mu}_j^y,\tilde{\sigma}}||_1 \leq 
2 \sum_{k=1}^{d_y} \bigg\{\frac{|\mu_{jk}^y - \tilde{\mu}_{jk}^y|}{\sigma \wedge \tilde{\sigma}} 
+ \frac{|\sigma - \tilde{\sigma}|}{\sigma \wedge \tilde{\sigma}} \bigg\}
\leq \frac{\epsilon}{4}.
\] 
Also,
\begin{align*}
& \frac{\sum_{j=1}^m \alpha_j | K_j - \tilde{K}_j|}{\sum_{j=1}^m \alpha_j K_j}  +   \frac{\sum_{j=1}^m \tilde{K}_j | \alpha_j - \tilde{\alpha}_j|}{\sum_{j=1}^m \alpha_j K_j} \\
&  \leq \max_{j} \frac{ | K_j - \tilde{K}_j |}{K_j}  + \max_{j} \frac{ | \alpha_j - \tilde{\alpha}_j|}{\alpha_j}   + \max_{j}  \frac{ | K_j - \tilde{K}_j |  | \alpha_j - \tilde{\alpha}_j|}{\alpha_jK_j}.
\end{align*}
Since $|\alpha_j - \tilde{\alpha}_j|/\alpha_j \leq \epsilon/12$ and $\epsilon<1$, 
the above display is bounded by $\epsilon/4$ if we can show $| K_j - \tilde{K}_j |/K_j \leq \epsilon /12$.
Observe that
\begin{align}
& \abs{\frac{||x- \mu_j^x||^2}{2 \sigma^2} - \frac{||x- \tilde{\mu}_j^x||^2}{2 \tilde{\sigma}^2}} 
\leq \frac{1}{2} \abs{\frac{1}{\sigma^2} -  \frac{1}{\tilde{\sigma}^2} }\norm{x - \mu_j^x}^2 
+  \frac{1}{2\tilde{\sigma}^2} \abs{\norm{x - \mu_j^x}^2 -  \norm{x - \tilde{\mu}_j^x}^2  } 
\nonumber
\\
& \leq \frac{||x- \mu_j^x||^2\abs{(\sigma -  \tilde{\sigma})/\sigma}}{\underline{\sigma}^2}  
+ \frac{||\mu_j^x - \tilde{\mu}_j^x||( 2||x|| +||\mu_j^x|| + ||\tilde{\mu}_j^x|| )}{2\underline{\sigma}^2}  
\leq \frac{\epsilon}{48} + \frac{\epsilon}{48} = \frac{\epsilon}{24},  
\label{eq:mu_tilde_mu}
\end{align}
where the penultimate inequality follows from $||x- \mu_j^x||^2\leq d_x$, 
$2||x|| +||\mu_j^x|| + ||\tilde{\mu}_j^x|| \leq 4 d_x^{1/2}$, and 
$||\mu_j^x - \tilde{\mu}_j^x|| \leq d_x^{1/2} \max_l |\mu_{jl}^x - \tilde{\mu}_{jl}^x|$.
Now since $|1- e^x| < 2|x|$ for  $|x| < 1$, 
\begin{eqnarray}
\label{eq:K_j_tildeK_j}
 \frac{ \abs{K_j - \tilde{K}_j }}{K_j} & = & \abs{1 -  \exp \bigg\{\frac{||x- \mu_j^x||^2}{2 \sigma^2} - \frac{||x- \tilde{\mu}_j^x||^2}{2 \tilde{\sigma}^2}  \bigg\}} \\
 & \leq & 2\abs{\frac{||x- \mu_j^x||^2}{2 \sigma^2} - \frac{||x- \tilde{\mu}_j^x||^2}{2 \tilde{\sigma}^2}} \leq \frac{\epsilon}{12}.
\nonumber
\end{eqnarray}
This concludes the proof for the covering number.

Next, let us obtain an upper bound for $\Pi(\mathcal{F}^c)$.  From the assumptions in Section \ref{sec:prior}
\[
\Pi(\exists \, j \in \{1, \ldots, m\}, \,\text{s.t.} \,\mu_j^y \in [-\overline{\mu}, \overline{\mu}]^{d_y}) \leq m \exp(-a_{13} \overline{\mu}^{\tau_3}).   
\]
For all sufficiently large $H$,
\[
\Pi(m > H) = C_1
\sum_{i=H+1}^\infty e^{- a_{10} i (\log i)^{\tau_1}} %\leq C_1 \int_{H}^\infty  e^{- a_{10} x (\log x)^{\tau_1}} dx 
\leq C_1 \int_{H}^\infty  e^{- a_{10} r (\log H)^{\tau_1}} dr \leq e^{- a_{10} H (\log H)^{\tau_1}}.
\]
Observe that $\alpha_j | m \sim  \mbox{Beta}(a/m, a(m-1)/m )$.  
Considering separately $a(m-1)/m -1<0$ and $a(m-1)/m -1 \geq 0$, it is easy to see that
$(1-q)^{a(m-1)/m -1} \leq 2$ for any $q \in [0,\underline{\alpha}]$ and $\underline{\alpha}\leq 1/2$.
Thus, 
 \begin{align}
 & \Pi(\alpha_j < \underline{\alpha}|m) = \frac{\Gamma(a)}{\Gamma(a/m)\Gamma(a(m-1)/m)} \int_{0}^{\underline{\alpha}} q^{a/m -1}  (1-q)^{a(m-1)/m -1}  dq  \nonumber \\
&\leq \frac{ \Gamma(a)}{\Gamma(a/m)\Gamma(a(m-1)/m)}  2 \int_{0}^{\underline{\alpha}} q^{a/m -1}    dq  \nonumber \\
& =  \frac{\Gamma(a)2 \underline{\alpha}^{a/m}}{\Gamma(a/m+1) \Gamma(a(m-1)/m)}  
\leq  e^2  2 \Gamma(a+1) \underline{\alpha}^{a/m}=C(a)\underline{\alpha}^{a/m}, \label{eq:alpha1}
 \end{align} 
where the final inequality is implied by the following facts: $\Gamma(a/m +1) \geq \int_{1}^{\infty} q^{a/m} e^{-q} dq \geq e^{-1}$ and $\Gamma(a(m-1)/m) \geq \int_{0}^1 q^{a(m-1)/m -1} e^{-q} dq \geq me^{-1}/a(m-1)$.  

Consider $\Pi (\sigma \notin [\underline{\sigma}, \overline{\sigma}])=\Pi(\sigma^{-1} \geq \underline{\sigma}^{-1}) 
+ \Pi(\sigma^{-1} \leq \overline{\sigma}^{-1})$.  
Since the prior for $\sigma$ satisfies  \eqref{eq:asnPrior_sigma1} and \eqref{eq:asnPrior_sigma2}, for 
 sufficiently large $\overline{\sigma}$ and small $\underline{\sigma}$
\begin{align}
\Pi(\sigma^{-1} \geq \underline{\sigma}^{-1}) \leq a_1\exp \{-a_2 \underline{\sigma}^{-2 a_3}\}, \quad 
\Pi(\sigma^{-1} \leq \overline{\sigma}^{-1})  \leq 
a_4 \overline{\sigma}^{-2a_5} = a_4 \exp\{-2a_5 \log \overline{\sigma}\}.
\end{align}

Now observe that 
\begin{align*}
\Pi(\mathcal{F}^c) \leq 
&  \; \Pi \left(\exists   m \leq H, \, \exists  j \leq m, \, \text{s.t.} \, 
             \mu_j^y \notin \{[-\overline{\mu}, \overline{\mu}]^{d_y}\}^c \right)   
	 + \Pi( m > H)  
\\ 
& +\Pi (\sigma \notin [\underline{\sigma}, \overline{\sigma}]) 
  + \Pi\left(\exists  m \leq H, \exists  j \leq m, \, \text{s.t.} \, \alpha_{j} <\underline{\alpha} \big | m\right)  
\\
& \leq   \sum_{m=1}^ H m\Pi( \mu_j^y \notin \{[-\overline{\mu}, \overline{\mu}]^{d_y}\}^c)  
  + \sum_{m=1}^ H m  \Pi( \alpha_{j} < \underline{\alpha} |m) 
	+ \Pi( m > H) + \Pi (\sigma \notin [\underline{\sigma}, \overline{\sigma}])  
	\\
& \leq   \frac{H(H+1)}{2} \exp\{-a_{13} \overline{\mu} ^{\tau_3}\}   
 +   \frac{H(H+1)}{2} C(a) \underline{\alpha}^{a/H}
 + \exp\{-0.5 a_{10} H(\log H)^{\tau_1}  \} 
 + \Pi (\sigma \notin [\underline{\sigma}, \overline{\sigma}]) 
\\
& \leq  H^2 \exp\{-a_{13} \overline{\mu} ^{\tau_3}\}  + H^2 \underline{\alpha}^{a/H} 
  +   \exp\{-a_{10} H(\log H)^{\tau_1}  \} 
	+ \Pi (\sigma \notin [\underline{\sigma}, \overline{\sigma}]).  
\end{align*}
%This concludes the proof of Proposition \ref{lm:sieve}.  
\end{proof}

\begin{theorem} 
\label{tm:sieve_n}
 For $n\geq 1$,
let $\epsilon_n = n^{- \beta/ (2\beta + d)}(\log n)^t$, $\tilde{\epsilon}_n = n^{- \beta/ (2\beta + d)}(\log n)^{t_0}$ for $ t_0 > (ds + \max\{\tau_1,1,\tau_2/\tau\}) / (2 + d/\beta)$  and define $\mathcal{F}_n$ as in \eqref{eq:sieve} with $\epsilon = \epsilon_n, H =
n\epsilon_n^2 /(\log  n)$, 
$\underline{\alpha}= e^{-n H}$,
$\underline{\sigma} = 	n^{-1/(2a_3)}$, $\overline{\sigma} = 	e^n$,
and $\overline{\mu}=n^{1/\tau_3}$. 
 Then for all  $t > t_0 +  \max \{0, (1- \tau_1)/2\}$, and some constants 
$c_1,c_3 >0$ and every $c_2 >0$, 
$\mathcal{F}_n$ satisfies \eqref{eq:entropy} and \eqref{eq:sievecomplement} for all large $n$.  \end{theorem}

\begin{proof}
Since $d_1 \leq d_{SS}$ and $d_h\leq d_1^2$, Theorem \ref{th:sieve} implies
\[
 \log J(\epsilon_n,\mathcal{F}_n,  \rho) 
\leq  c_1 H \log n =  c_1 n\epsilon_n^2.  
\]
Also,
\begin{align*}
 \Pi(\mathcal{F}_n^c) 
& \leq  H^2 \exp\{-a_{13} n\}  
   + H^2 \exp\{-a n\} 
   + \exp\{- a_{10} H(\log H)^{\tau_1} \} 
	%+ \exp\{-b_4 n\} + \exp\{-b_5 n\}
\\
&
+ a_1\exp \{-a_2 n\} 
+ 
a_4 \exp\{-2a_5 n\}
.  
\end{align*}
Hence, $\Pi(\mathcal{F}_n^c) \leq e^{-(c_2+4) n \tilde{\epsilon}_n^2}$ for any $c_2$ if 
$\epsilon_n^2 (\log n)^{\tau_1 -1}/\tilde{\epsilon}_n^2 \rightarrow \infty$, which holds for 
$t > t_0 +  \max \{0, (1- \tau_1)/2\}$.

\end{proof}

\section{Unbounded covariate space}
\label{sec:non_comp}
The assumption of bounded covariate space $\mathcal{X}$ in Section \ref{sec:main_results} could be restrictive  in some applications.  In this section, we consider a generalization of our result to the case when the covariate space is possibly unbounded.  
We re-formulate the assumptions on the data generating process and the prior distributions below. 

\subsection{Assumptions about data generating process}
\label{sec:asns_nonc}

Let $\mathcal{X}  \subset \mathbb{R}^{d_x}$.  
First, let us assume that there exist a constant $\eta > 0$ and
a probability density function $\bar{g}_0(x)$ with respect to the Lebesgue measure such that $\eta \bar{g}_0(x) \geq g_0(x)$ for all $x \in \mathcal{X}$ and $\tilde{f}_0(y, x) = f_0(y|x) \bar{g}_0(x) \in  \mathcal{C}^{\beta,L,\tau_0}$.  
Second, we assume $g_0$ satisfies
\begin{eqnarray}\label{ass_g0_nonc}
\int e^{\kappa \norm{x}^2} g_0(x) dx  \leq B < \infty
\end{eqnarray}
for some constant $\kappa > 0$.  
Third, $\tilde{f}_0(y, x)$ is assumed to satisfy 
\begin{equation}
	\label{eq:asnE0Dff0_Lf0tilde}
	\int_{\mathcal{Z}} \left|\frac{D^k \tilde{f}_0(y,x)}{\tilde{f}_0(y, x)}\right|^{(2\beta+\varepsilon)/k} \tilde{f}_0(y, x) dydx< \infty,
	\;
	\int_{\mathcal{Z}} \left|\frac{L(y,x)}{\tilde{f}_0(y,x)}\right|^{(2\beta+\varepsilon)/\beta} \tilde{f}_0 (y, x) dydx< \infty
	\end{equation} 
	for all $k \leq \lfloor \beta \rfloor$ and some $\varepsilon>0$.
Finally, for all sufficiently large $(y,x) \in \mathcal{Y} \times \mathcal{X}$ and some positive $(c,b,\tau)$, 
\begin{equation}
\label{eq:asntildef0_exp_tails2}
\tilde{f}_0(y, x) \leq c \exp(-b ||(y,x)||^\tau). 
	\end{equation} 
Let us elaborate on how the above assumptions allow for $f_0$ of smoothness level 
$\beta$.  First of all, the original assumptions on the data generating process for the bounded $\mathcal{X}$ are a special case of the assumptions here with $\bar{g}_0$ being a uniform density on $\mathcal{X}$.
Second, when covariate density $g_0$ has a higher smoothness level than $\beta$ and 
$\bar{g}_0=g_0$, the assumption $\tilde{f}_0 \in  \mathcal{C}^{\beta,L,\tau_0}$ essentially restricts the smoothness of $f_0$ only.
Finally, when $g_0$ has a lower smoothness level than $\beta$, then our assumptions require existence of a sufficiently smooth and well behaved upper bound on $g_0$ in addition to $f_0$ having smoothness level $\beta$.

\subsection{Prior} \label{sec:prior_nonc} 
The assumption on the prior for $\mu_j^x$ is the only part of the prior from Section \ref{sec:prior} that we need to modify here.
Similarly to the prior on $\mu_j^y$, we assume that the prior density for $\mu_j^x$ is bounded  below for some $a_{14}, a_{15}, \tau_4 > 0$ by
\[
a_{14} \exp(-a_{15} ||\mu_j^x||^{\tau_4} ),
\]
and for some $a_{16}, \tau_5>0$ and all sufficiently large $r > 0$,
\begin{equation*}
	\label{eq:asnPrior_mux_tail_ub}
		1- \Pi(\mu_{j}^x \in [-r,r]^{d_x}) \geq \exp(-a_{16} r^{\tau_5}). 
\end{equation*}
%Note that for simplicity of exposition, the constants $\tau_2, \tau_3$ are taken to be the same as that for the prior for $\mu_j^y$. 

\begin{corollary}
\label{cr:nonc}
Under the assumptions in Sections \ref{sec:asns_nonc}-\ref{sec:prior_nonc}, 
the posterior contracts at the rate specified in Theorem \ref{th:rate_cond_dens}. 
\end{corollary}
\begin{proof}
We will first show that an analog of the
{\em prior thickness} result from Theorem \ref{th:prior_thickness}
holds with the same choice of 
$\tilde{\epsilon}_n$.   By Corollary \ref{cor:dH_cond_bdd_dHjoint} in the appendix, for $p(\cdot | \cdot, \theta, m)$ defined in \eqref{eq:joint_mix_def},  
\begin{eqnarray}
d_h^2 (f_0, p(\cdot | \cdot, \theta, m)) \leq  
C_1 d_H^2(f_0 \bar{g}_0, p(\cdot |\theta, m) ).  
\end{eqnarray}
Since the joint density $f_0 \bar{g}_0$ satisfies the assumptions of Theorem 4 in \cite{ShenTokdarGhosal2013}, the rest of the proof of the prior thickness result is exactly the same as the proof of Theorem \ref{th:prior_thickness} except for $\pi_\mu$ in \eqref{eq:Pr_muU} would now denote the joint prior density of $(\mu_j^y,\mu_j^x)$.  

Next, we will construct an appropriate {\em sieve}. For  sequences $\underline{\alpha},  \overline{\mu}, \overline{\mu}^x, H, \underline{\sigma}, \overline{\sigma}$ to be chosen later, define:
\begin{eqnarray*}%\label{eq:sieve}
\mathcal{F} = \{p(y|x,\theta,m): \; m\leq H, \; \alpha_j \geq \underline{\alpha}, \; 
\sigma \in [\underline{\sigma}, \overline{\sigma}], \mu_j^y \in [-\overline{\mu}, \overline{\mu}]^{d_y}, \mu_j^x \in [-\overline{\mu}^x, \overline{\mu}^x]^{d_x}, \,j=1, \ldots,m\}.
\end{eqnarray*}
The choice of $S_{\mathcal{F}}$, an $\epsilon$-net of $\mathcal{F}$ is same as in the proof of Theorem \ref{th:sieve}  with the following modifications.  $S_{\mu^x}^m$ now contains centers of $|S_{\mu^x}^m| =  \lceil 192d_x (\overline{\mu}^x)^2/(\underline{\sigma}^2 \epsilon) \rceil$ equal length intervals 
partitioning $[-\overline{\mu}^x, \overline{\mu}^x]$. We also need an adjustment to $S_{\sigma}$ here: 
\begin{eqnarray*}
S_{\sigma} =  \{\sigma^{l},l=1, \ldots, N_{\sigma} = \lceil{ \log (\overline{\sigma}/\underline{\sigma})/ (\log (1 + \underline{\sigma}^2\epsilon/(384(\overline{\mu}^x)^2\max\{d_x,d_y\})}\rceil, \sigma^1= \underline{\sigma}, \\ (\sigma^{l+1} - \sigma^l)/\sigma^l = \underline{\sigma}^2\epsilon/(384(\overline{\mu}^x)^2\max\{d_x,d_y\})\}.  
\end{eqnarray*}

 Since we are dealing with possibly unbounded $\mathcal{X}$ here, we will find the covering number of $\mathcal{F}$ in $d_1$ instead of $d_{SS}$.   The only part different from the proof of Theorem \ref{th:sieve}  is the treatment of
$\abs{K_j - \tilde{K}_j }/K_j$.
To show that $\int \abs{K_j - \tilde{K}_j }/K_j g_0(x) dx \leq \epsilon/12$, we 
divide the range of integration into two parts: $\mathcal{X}_1=\{x \in 
\mathcal{X}:\, \abs{x_l} \leq \overline{\mu}^x, \, l = 1, \ldots, d_x\}$ and $\mathcal{X}\setminus \mathcal{X}_1$.

For $x \in \mathcal{X}_1$, the same argument as in the bounded covariate space (inequalities in \eqref{eq:mu_tilde_mu} and \ref{eq:K_j_tildeK_j}) combined with
$\norm{x} \leq d_x^{1/2}\overline{\mu}^x$,
$||x-\mu_j^x||^2\leq 4 (\overline{\mu}^x)^2 d_x$, 
$2||x|| +||\mu_j^x|| + ||\tilde{\mu}_j^x|| \leq 4 \overline{\mu}^x d_x^{1/2}$, 
$|\sigma - \tilde{\sigma}|/\sigma \leq \underline{\sigma}^2\epsilon/(384(\overline{\mu}^x)^2 d_x)$
and
$|\mu_{jl} - \tilde{\mu}_{jl}| \leq \epsilon \underline{\sigma}^2 / (192 d_x \overline{\mu}^x)$
imply 
$\abs{K_j - \tilde{K}_j }/K_j \leq \epsilon/24$.

For $x \in \mathcal{X} \setminus \mathcal{X}_1$,
the left hand side of \eqref{eq:mu_tilde_mu} is bounded above by
\begin{align}
& (2||x||^2 + 2||\mu_j^x||^2) \epsilon/(384(\overline{\mu}^x)^2 d_x) 
+
(||x|| +   d_x^{1/2} \overline{\mu}^x) d_x^{1/2}\epsilon/(192\overline{\mu}^x d_x) 
\\
& \leq \epsilon/96 + ||x||^2 \epsilon/(192(\overline{\mu}^x)^2 d_x) + 
||x|| \epsilon/(192\overline{\mu}^x d_x^{1/2})
\\
& \leq \epsilon/96 + ||x||^2 \epsilon/(96\overline{\mu}^x d_x^{1/2}),
\end{align}
where the last inequality holds for $\overline{\mu}^x \geq 1$ as $||x|| \geq \overline{\mu}^x$ for $x \in \mathcal{X} \setminus \mathcal{X}_1$. Since $\abs{1- e^r} \leq e^{\abs{r}}$ for all $r \in \mathbb{R}$,
\[
\abs{ \frac{K_j  - \tilde{K}_j}{K_j}} \leq  \exp \bigg(\frac{\epsilon}{96}+ \frac{\epsilon \norm{x}^2}{96 d_x^{1/2} \overline{\mu}^x} \bigg), \; \forall x \in \mathcal{X} \setminus \mathcal{X}_1.
\]

Now, 
\begin{eqnarray*}
\int \abs{ \frac{K_j  - \tilde{K}_j}{K_j}}  g_0(x) dx &\leq&  \int_{\mathcal{X}_1}  \abs{ \frac{K_j  - \tilde{K}_j}{K_j}} g_0(x) dx + \int_{\mathcal{X} \setminus \mathcal{X}_1}  \abs{ \frac{K_j  - \tilde{K}_j}{K_j}} g_0(x) dx  \\
&\leq& \frac{\epsilon}{24}  + \exp \bigg(\frac{\epsilon}{96} \bigg)\int_{\mathcal{X} \setminus \mathcal{X}_1}   
\exp \bigg( \frac{\epsilon \norm{x}^2}{96 d_x^{1/2} \overline{\mu}^x} \bigg) g_0(x) dx  \\
&\leq& \frac{\epsilon}{24}  + \exp \bigg(\frac{\epsilon}{96} \bigg)\int_{\mathcal{X} \setminus \mathcal{X}_1} \exp \big(-\kappa_{\epsilon} \norm{x}^2\big)  \exp (\kappa  \norm{x}^2) g_0(x) dx
\end{eqnarray*}
where $\kappa_{\epsilon}  =  \kappa -  \epsilon /(96 d_x^{1/2} \overline{\mu}^x) \geq \kappa/2$ for small $\epsilon$ and large 
$\overline{\mu}^x$.   Since $\norm{x} \geq  \overline{\mu}^x$ in $\mathcal{X} \setminus \mathcal{X}_1$, we have 
\begin{eqnarray*}
\int \abs{ \frac{K_j  - \tilde{K}_j}{K_j}}  g_0(x) dx &\leq&  \frac{\epsilon}{24}  + \exp \bigg(\frac{\epsilon}{96}\bigg)   \exp \big(-\kappa(\overline{\mu}^x)^2/2\big)
\int_{\mathcal{X} \setminus \mathcal{X}_1}   \exp (\kappa  \norm{x}^2) g_0(x) dx \\
&\leq& \frac{\epsilon}{24}  +  B \exp \bigg(\frac{\epsilon}{96}\bigg) \exp \big(-\kappa(\overline{\mu}^x)^2/2\big),
\end{eqnarray*}
where $B$ is defined in \eqref{ass_g0_nonc}.
For $(\overline{\mu}^x)^2 \geq -(2/\kappa) \log \{ \epsilon e^{-\epsilon/ 96} / 24 B \}$,  
\begin{eqnarray*}
\int \abs{ \frac{K_j  - \tilde{K}_j}{K_j}}  g_0(x) dx \leq  \frac{\epsilon}{12}. 
\end{eqnarray*}
Hence for $(\overline{\mu}^x)^2 \geq -(2/\kappa) \log \{ \epsilon e^{-\epsilon/ 96} / 24 B \}$,   following the proof of Theorem \ref{th:sieve} we obtain, 
\begin{align*}
J(\epsilon, \mathcal{F}, d_1) \leq & 
H \cdot \left \lceil \frac{16\overline{\mu}d_y}{\underline{\sigma}\epsilon} \right \rceil^{Hd_y}
\cdot \left \lceil \frac{192 d_x (\overline{\mu}^x)^2}{\underline{\sigma}^2 \epsilon} \right \rceil^{Hd_x}
\cdot H \left \lceil \frac{ \log (\underline{\alpha}^{-1}) }{\log (1 + \epsilon/[12H])} \right \rceil^{H-1}
\\ 
& \cdot \left \lceil \frac{ \log (\overline{\sigma}/\underline{\sigma})} {\log (1 + \underline{\sigma}^2\epsilon/[384 (\overline{\mu}^x)^2 \max\{d_x,d_y\}])} \right \rceil.
\end{align*}
Observe that $\Pi(\mathcal{F}^c)$ is bounded above by 
\[
  H^2 \exp\{-a_{13} \overline{\mu} ^{\tau_3}\} + H^2 \exp\{-a_{16} (\overline{\mu}^x)^{\tau_5}\}  + H^2 \underline{\alpha}^{a/H} 
  +   \exp\{-a_{10} H(\log H)^{\tau_1}  \} 
	+ \Pi (\sigma \notin [\underline{\sigma}, \overline{\sigma}]). 
\]
The rest of the proof follows the argument in the proof of Theorem \ref{tm:sieve_n}
with the same sequences and $\overline{\mu}^x = n^{1/\tau_5}$. 

\end{proof}

\section{Irrelevant covariates}
\label{sec:irr_cov}

In applications, researchers often tackle the problem of selecting a set of relevant covariates
for regression or conditional distribution estimation.
In the Bayesian framework, this is usually achieved by introducing latent indicator variables for inclusion 
of covariates in the model, see, for example, \cite{bhattacharya2014anisotropic}, \cite{shen2014adaptive}, \cite{yang2014minimax}.  This is equivalent to a Bayesian model averaging procedure, where every possible subset of covariates represents a model.
It is straightforward to extend the results of the previous sections to a model with latent indicator variables for covariate inclusion  and show that the 
posterior contraction rate will not be affected by the irrelevant covariates.
In this section, we show that even without introduction of the indicator variables, 
irrelevant covariates do not affect the posterior contraction rate in a version of our model with component specific scale parameters.

Let
$\theta=
\{
\mu_j^y,\mu_{j}^{x}, \alpha_j, j=1,2,\ldots; \; \sigma^y=(\sigma^y_1,\ldots,\sigma^y_{d_y}), \sigma^x=(\sigma^x_1,\ldots,\sigma^x_{d_x})\}$ and
\begin{equation*}
\label{eq:cond_mix_def_diffsigma}
p(y|x, \theta, m) = \sum_{j=1}^m\frac{ \alpha_j \exp\{-0.5\sum_k (x_k-\mu_{jk}^x)^2/(\sigma^x_k)^2 \}   
}
{\sum_{i=1}^m \alpha_i \exp\{-0.5\sum_k (x_k-\mu_{ik}^x)^2/(\sigma^x_k)^2 \}   
}\phi_{\mu_j^y,\sigma^y}(y).
\end{equation*}
Suppose $f_0$ depends only on the first $d_x^0<d_x$ covariates $x_{1d_x^0}=(x_1,\ldots,x_{d_x^0})$ with the marginal density $g_{1d_x^0}$.
Let us assume conditions \eqref{ass_g0_nonc}-\eqref{eq:asntildef0_exp_tails2} from Section \ref{sec:asns_nonc} with the following change in the definition of $\tilde{f}_0$:
for $\eta > 0$ and
a probability density function $\bar{g}_{1d_x^0}(x_{1d_x^0})$ with respect to the Lebesgue measure such that $\eta \bar{g}_{1d_x^0}(x_{1d_x^0}) \geq g_{1d_x^0}(x_{1d_x^0})$ for all $x \in \mathcal{X}$ and $\tilde{f}_0(y, x_{1d_x^0}) = f_0(y|x_{1d_x^0}) \bar{g}_{1d_x^0}(x_{1d_x^0}) \in  \mathcal{C}^{\beta,L,\tau_0}$.
In addition, let us assume that the tail condition \eqref{eq:asntildef0_exp_tails2} holds for $f_0g_0$.

For $l=1,\ldots,d_y$ and $k=1,\ldots,d_x$, 
$\sigma^y_l$ and $\sigma^x_k$ are assumed to be independent a priori with densities satisfying
\eqref{eq:asnPrior_sigma1}-\eqref{eq:asnPrior_sigma3}.  Other parts of the prior are assumed to be the same as in Section
\ref{sec:prior_nonc}.

Let us briefly explain why we introduce component specific scale parameters.
Our proof of the following corollary exploits the fact that
when $\mu_{j k}^{x}=0$ for $k > d_x^0$ and all $j$, covariates $x_k$ for $k > d_x^0$ do not enter the model, 
and for $\mu_{j k}^{x}$ near zero and large $\sigma^x_k$ for $k > d_x^0$ this holds approximately.  
At the same time, approximation arguments in the prior thickness results require $\sigma^x_k$ very close to zero for $k \leq d_x^0$.  
Thus, we need to allow scale parameters for relevant and irrelevant covariates to take different values 
(our assumption of different scale parameters for components of $y$ is not essential for the result).

\begin{corollary}
\label{cr:irr_cov}
Under the assumptions of this section, the posterior contracts at the rate specified in Theorem \ref{th:rate_cond_dens} with $d=d_y+d_x$ replaced by $d^0=d_y+d_x^0$.
\end{corollary}
\begin{proof}
 First, consider the prior thickness result in Theorem \ref{th:prior_thickness}.
For any $\theta$ let 
\[
\theta_{d_x^0}=\{
\mu_j^y, \mu_{j1}^{x}, \ldots,\mu_{j d_x^0}^{x}, 
\alpha_j, j=1,2,\ldots; \; \sigma^y, \sigma^x_1,\ldots,\sigma^x_{d_x^0}
\}
\]
and define $p(\cdot | \cdot, \theta_{d_x^0},m)$ and $p(\cdot | \theta_{d_x^0},m)$ as before with $(y,x_{1d_x^0})$ as the arguments.
By the triangle inequality, 
\begin{align*}
d_h(f_0, p (\cdot | \cdot, \theta,m)) \leq d_h(f_0, p(\cdot | \cdot, \theta_{d_x^0},m))  + d_h(p(\cdot | \cdot, \theta_{d_x^0},m), p(\cdot | \cdot, \theta,m)). 
\end{align*}
By Corollary \ref{cor:dH_cond_bdd_dHjoint} in the appendix, $d_h(f_0, p(\cdot | \cdot, \theta_{d_x^0})) \leq C_1 
d_H(f_0 \bar{g}_{1d_x^0}, p(\cdot |  \theta_{d_x^0}))$.  By the argument leading to \eqref{eq:f0upsbeta} and \eqref{eq:P0_zgeqa}, there exist
$\theta_{d_x^0}^\star$ such that 
$d_H(f_0 \bar{g}_{1d_x^0}, p(\cdot |\theta_{d_x^0}^\star, m) ) \leq C_2 \sigma_n^\beta$,
$P_0(\norm{y,x} > a_{\sigma_n}) \leq B_0 \sigma_n^{4\beta + 2\epsilon +8}$, 
$z_j$ and $U_j$ are defined on the space for 
$(y,x_{1d_x^0})$, and
$1 \leq N < K \leq C_2 \sigma_n^{-d^0} \{\log (1/ \tilde{\epsilon}_n) \}^{d^0 +d^0/\tau}$.
Let
\begin{align*}
	S_{\theta^\star}=&\big \{
	(\mu_j, \alpha_j, \, j=1,2,\ldots; \sigma^y, \sigma^x): 
	 \; (\mu_j^y,\mu_{j1}^{x}, \ldots,\mu_{j d_x^0}^{x})  \in U_j,\; \\
	&
	||(\mu_{j d_x^0+1}^{x}, \ldots,\mu_{j d_x}^{x})|| \leq \sigma_n \tilde{\epsilon}_n^{2b_1}, \;
	j\leq K; \\
	& \sum_{j=1}^K \abs{\alpha_j - \alpha_j^\star} \leq 2\tilde{\epsilon}_n^{2 d^0 b_1}, \, \min_{j=1, \ldots, K} \alpha_j \geq \tilde{\epsilon}_n^{4 d^0 b_1}/2; 
	\\	& 
	(\sigma^x_k)^2, (\sigma^y_l)^2 \in [ \sigma_n^{2}/(1+ \sigma_n^{2\beta}), \sigma_n^{2} ],\;	l \leq d_y, \; k \leq d_x^0;\\
	& (\sigma^x_k)^2 \in [ a_{\sigma_n}^2,  2a_{\sigma_n}^2],\; 	k=d_x^0+1, \ldots, d_x
	\big \}.
\end{align*}
For $\theta \in S_{\theta^\star}$ and $m = K$, as in the proof of Theorem \ref{th:prior_thickness}, 
\begin{align}\label{eq:diff}
d_H(f_0 \bar{g}_{1d_x^0}, p(\cdot |\theta_{d_x^0}, m) )   \leq  d_H(f_0 \bar{g}_{1d_x^0}, p(\cdot |\theta_{d_x^0}^*, m) )   +  
d_H(p(\cdot |\theta_{d_x^0}, m), p(\cdot |\theta_{d_x^0}^*, m) ) \leq C_3 \sigma_n^\beta. 
\end{align}
Next, we tackle $d_h(p(\cdot | \cdot, \theta_{d_x^0},m), p(\cdot | \cdot, \theta,m))$.   Following  the entropy
calculations in Theorem \ref{th:sieve},  we have for $m=K$, 
\[
d_h^2(p(\cdot | \cdot, \theta_{d_x^0},m), p(\cdot | \cdot, \theta,m)) \leq \int \max_{1\leq j \leq K}  
|K_j - \tilde{K}_j|/\abs{K_j} g_0(x) dx,\] 
where  
\begin{align*}
K_j = \exp \bigg\{-\sum_{k=1}^{d_x}  \frac{(x_k - \mu_{jk}^x)^2}{2(\sigma_k^x)^2} \bigg\}, \quad 
\tilde{K}_j = \exp \bigg\{-\sum_{k=1}^{d_x^0}  \frac{(x_k - \mu_{jk}^x)^2}{2(\sigma_k^x)^2}  
- \sum_{k = d_x^0 +1}^{d_x} \frac{x_k^2}{2(\sigma_k^x)^2}\bigg\}, 
\end{align*}
and $\tilde{K}_j$ is normalized in a convenient way.
To show that $\int |K_j - \tilde{K}_j |/K_j g_0(x) dx \leq 2\sigma_n^{2\beta}$, we 
divide the range of integration into two parts: $\mathcal{X}_1=\{x \in 
\mathcal{X}:\, \abs{x_l} \leq A_n, \, l = d_x^0+1, \ldots, d_x\}$ and $\mathcal{X}\setminus \mathcal{X}_1$, where $A_n = a_{\sigma_n}^2 \log (B/\sigma_n^{2\beta})$
and $B$ is defined in assumption \eqref{ass_g0_nonc}.    
Observe that for $\theta \in S_{\theta^\star}$, $x \in \mathcal{X}_1$, and all sufficiently large $n$,  
\begin{align*}
\abs{\sum_{k=d_x^0+1}^{d_x}  \frac{(2x_k - \mu_{jk}^x)(-\mu_{jk}^x)}{2(\sigma_k^x)^2}}  \leq 
 \frac{2A_n(d_x - d_x^0)\sigma_n \tilde{\epsilon}_n^{2b_1}}{a_{\sigma_n}^2} \leq 1
\end{align*}
and, hence, using $\abs{1- e^r} \leq \abs{r}$ for $\abs{r} \leq 1$, we obtain for 
$\theta \in S_{\theta^\star}$ and $x \in \mathcal{X}_1$, 
\begin{align}\label{inc:eq1}
\abs{\frac{K_j - \tilde{K}_j}{K_j}}  =  \abs{1- 
\exp \bigg\{\sum_{k=d_x^0+1}^{d_x}  \frac{(2x_k - \mu_{jk}^x)(-\mu_{jk}^x)}{2(\sigma_k^x)^2} \bigg\}} \leq  \frac{2A_n(d_x - d_x^0)\sigma_n \tilde{\epsilon}_n^{2b_1}}{a_{\sigma_n}^2} \leq \sigma_n^{2\beta}. 
\end{align}
For $x \in \mathcal{X}\setminus \mathcal{X}_1$ using $\abs{1- e^r} \leq e^{\abs{r}}$, 
\begin{align*}
\int_{ \mathcal{X}\setminus \mathcal{X}_1} \max_{1\leq j \leq K} \abs{\frac{K_j - \tilde{K}_j}{K_j}}  g_0(x) dx  
&\leq \int_{ \mathcal{X}\setminus \mathcal{X}_1}  e^{\sum_{k=d_x^0+1}^{d_x} \frac{\abs{x_k}}{a_{\sigma_n}^2} - \kappa \norm{x}^2}  e^{\kappa \norm{x}^2} g_0(x) dx.
\end{align*}
For $x \in \mathcal{X}\setminus \mathcal{X}_1$,  $\kappa \norm{x}^2 \geq \kappa (d_x - d_x^0)^{-1}(\sum_{k=d_x^0+1}^{d_x} \abs{x_k})^2 \geq 2\sum_{k=d_x^0+1}^{d_x} \abs{x_k}/a_{\sigma_n}^2$ and hence 
\begin{align}\label{inc:eq2}
\int_{ \mathcal{X}\setminus \mathcal{X}_1} \max_{1\leq j \leq K} \abs{\frac{K_j - \tilde{K}_j}{K_j}}  g_0(x) dx  \leq Be^{-A_n/a_{\sigma_n}^2} \leq \sigma_n^{2\beta}. 
 \end{align}
From \eqref{inc:eq1} and  \eqref{inc:eq2}, it follows that $d_h(p(\cdot | \cdot, \theta_{d_x^0}), p(\cdot | \cdot, \theta)) \leq 2^{1/2}\sigma_n^\beta$.

Next let us establish an analog of \eqref{eq:lambda_def} when $\norm{(y, x)} \leq a_{\sigma_n}$.
Using the argument leading to \eqref{eq:lambda_def} 
with $\norm{(y,x_{1d_x^0})} \leq a_{\sigma_n}$
and $((x_k - \mu_{jk}^x)/\sigma^x_k)^2 \leq 4$ for $k =d_x^0+1, \ldots, d_x$, $j=1, \ldots, m$,
we get for $\theta \in S_{\theta^\star}$ and $\norm{(y, x)} \leq a_{\sigma_n}$,
\begin{align*}
\frac{p(y | x, \theta, m)}{f_0(y|x)} \geq  
\frac{1}{f_0(y|x)} \min_{j=1, \ldots, K} \alpha_j \cdot \sigma_n^{-d_y} 
\exp\left\{-\sum_{k=d_x^0+1}^{d_x}  \frac{(x_k - \mu_{jk}^x)^2}{2(\sigma_k^x)^2}\right\}
\geq C_5 \tilde{\epsilon}_n^{4 d^0 b_1} \sigma_n^{-d_y} = \lambda_n. 
\end{align*}

For $\norm{(y, x)} \geq a_{\sigma_n}$,  
\begin{align*}
p(y \mid x, \theta, m) \geq \min_{1\leq j \leq m} C_6\sigma_n^{-d_y} \exp \bigg\{- \frac{||y - \mu_j^y||^2}{2\sigma_n^2}\bigg\} \geq C_7 \sigma_n^{-d_y} 
\exp\bigg(- C_8\frac{a_{\sigma_n}^2}{\sigma_n^2}  - C_{9} \frac{\norm{y}^2}{\sigma_n^2} 
\bigg)
\end{align*}
implies
\begin{align*}
\bigg\{\log \frac{f_0(y|x)}{p(y|x, \theta, m)}\bigg\}^2 \leq C_{10} \bigg(\frac{a_{\sigma_n}^4}{\sigma_n^4} + \frac{\norm{y}^4}{\sigma_n^4} \bigg). 
\end{align*} 
Then, following the proof of Theorem \ref{th:prior_thickness}, 
\begin{align*}
& \int \bigg\{\log \frac{f_0(y|x)}{p(y|x, \theta, m}\bigg\}^2 1\bigg\{ \frac{p(y|x, \theta, m}{f_0(y|x)}  < \lambda_n\bigg\} f(y|x) g_0(x) dydx 
\\
& \leq C_{11} \bigg\{ \frac{a_{\sigma_n}^4 P_0(\norm{Z} > a_{\sigma_n}) }{\sigma_n^4}   +   
\frac{ E_0(\norm{Y}^8)^{1/2} (P_0(\norm{Z} > a_{\sigma_n}))^{1/2}}{\sigma_n^4}
\bigg\}  
\leq C_{12} \sigma_n^{2\beta + \varepsilon/2}\sigma_n^{\varepsilon/2} a_{\sigma_n}^4 \\
& \leq   C_{12} \sigma_n^{2\beta + \varepsilon/2}.
\end{align*}
The rest of 
the proof of
$E_0 (\log (f_0(Y|X)/p(Y|X, \theta, m)))\leq A\tilde{\epsilon}_n^2$ and $E_0 ([\log (f_0(Y|X)/p(Y|X, \theta, m))]^2) \leq A\tilde{\epsilon}_n^2$ goes through without any changes.
   
The lower  bound for the prior probability of $S_{\theta^\star}$ and $m=K$ is the same as the one in 
Theorem \ref{th:prior_thickness}, except $d$ is replaced with $d^0$.  The only additional calculation for $\sigma_k^x$,
$k=d_x^0+1,\ldots,d_x$ follows from Assumption \eqref{eq:asnPrior_sigma3},
\begin{equation*}
\Pi((\sigma_k^x)^{-2} \in [ a_{\sigma_n}^{-2}/2, a_{\sigma_n}^{-2} ]) \gtrsim
a_{\sigma_n}^{-2a_7}. 
\end{equation*}

In the definition of sieve 
\eqref{eq:sieve}, let us replace  condition $\sigma \in [\underline{\sigma}, \overline{\sigma}]$
by
\[\sigma_l^y, \sigma^x_k \in [\underline{\sigma}, \overline{\sigma}],\; l=1,\ldots,d_y, \; k=1,\ldots,d_x.\]
The presence of the component specific scale parameters and the dimension of $x$ affect only constants 
in the sieve entropy bound and the bound on the prior probability of the sieve complement.
Thus, Theorem \ref{tm:sieve_n} holds with $d$ replaced by $d^0$.

\end{proof}

%\[
%\phi_{\mu_J,\sigma}(z) \geq \phi_{(\mu^y_J, \mu^x_{J1},\ldots,\mu^x_{Jd_x^0}),
%(\sigma^y,\sigma^x_1,\ldots,\sigma^x_{d_x^0})}(y,x_1,\ldots,x_d_{x}^0)) \cdot
%\phi_{(\mu^x_{Jd_x^0+1},\ldots,\mu^x_{Jd_x}),
%(\sigma^x_{d_x^0+1},\ldots,\sigma^x_{d_x})}(x_{d_x^0+1},\ldots,x_{d_{x}}))
%\]

\section{Finite sample performance}
\label{sec:simul}

In this section, we evaluate the finite sample performance of our conditional density model in Monte Carlo simulations.
Specifically, we explore how the sample size and irrelevant covariates affect estimation results.
We also compare our estimator with a conditional density kernel estimator from \cite{Hall2004} 
that is based on a cross-validation method for obtaining the band-width.  \cite{Hall2004} showed that irrelevant covariates do not affect 
the convergence rate of their estimator, and, thus, this estimator appears to be a suitable benchmark.
The kernel estimation results are obtained by the publicly available R package \texttt{np} (\cite{Hayfield2008}).

It has been established in the literature (see, for example, \cite{VillaniKohnGiordani:07}), that slightly more general 
specifications of covariate dependent mixture models perform better in practice.  Thus, we use the following specification
\begin{equation}
\label{eq:cond_mix_finsam}
p(y|x, \theta, m) = \sum_{j=1}^m\frac{ \alpha_j \exp\{-0.5\sum_{k=1}^{d_x}(x_k-\mu_{jk}^x)^2/(\sigma_{k}^x s^x_{jk})^2 \}   
}
{\sum_{i=1}^m \alpha_i \exp\{-0.5\sum_{k=1}^{d_x}(x_k-\mu_{ik}^x)^2/(\sigma^x_{k} s^x_{ik})^2 \}   
}\phi_{x^\prime \beta_j,\sigma^y s^y_j}(y),
\end{equation}
where $x^\prime \beta_j$ are used instead of locations $\mu_j^y$, $x$ includes zeroth coordinate equal 1, and local scale parameters $(s^y_j,s^x_{jk})$ introduced in addition to the global $(\sigma^y,\sigma^x_{k})$.
The prior is specified as follows,
\begin{align*}
&\beta_j  \stackrel{iid}{\sim} N(\ul{\beta}, \ul{H}_\beta^{-1}), \; 
\mu_j   \stackrel{iid}{\sim} N(\ul{\mu}, \ul{H}_\mu^{-1}), \\
&(s^y_j)^{-2}  \stackrel{iid}{\sim} G(\ul{A}_{s y}, \ul{B}_{s y}), \;   
(s^x_{jk})^{-2} \stackrel{iid}{\sim} G(\ul{A}_{s xk}, \ul{B}_{s xk }), \; k=1,\ldots,d_x, \\ 
&(\sigma^y)^{-1}  \stackrel{iid}{\sim} G(\ul{A}_{\sigma y}, \ul{B}_{\sigma y}), \;   
(\sigma^x_{k})^{-1}  \stackrel{iid}{\sim} G(\ul{A}_{\sigma xk}, \ul{B}_{\sigma xk}), \; k=1,\ldots, d_x, \\ 
&(\alpha_1,\ldots,\alpha_m)|m  \stackrel{iid}{\sim} D(\ul{a}/m,\ldots,\ul{a}/m), \\
&\Pi\big(m=k\big)  = (e^{\ul{A}_m}-1) e^{-\ul{A}_m \cdot k},
\end{align*}
where $G(A,B)$ stands for a Gamma distribution with shape $A$ and rate $B$.
It is shown in Theorem \ref{thm:lincoeflocals} in Appendix that an analog of Corollary \ref{cr:irr_cov} holds for this slightly more general setup.
To obtain estimation results for this model we use an MCMC algorithm developed in \cite{Norets:15}.

We use the following (data-dependent) values for prior hyper-parameters: 
\begin{align*}
&\ul{\beta} = \left(\sum_i x_i x_i^\prime\right)^{-1} \sum_i x_i y_i, \;  
\ul{H}_\beta^{-1} = \ul{c}_\beta\left(\sum_i x_i x_i^\prime\right)^{-1} \sum_i(y_i - x_i^\prime \ul{\beta})^2/n,\\
&\ul{\mu} = \sum_i x_i/n, \; \ul{H}_\mu^{-1} = \sum_i (x_i-\ul{\mu})(x_i-\ul{\mu})^\prime/n, \\
& \ul{A}_{\sigma y} = \ul{c}_\sigma/\left(\sum_i(y_i - x_i^\prime \ul{\beta})^2/n\right), \; 
\ul{B}_{\sigma y} = \ul{c}_\sigma/\left({\sum_i(y_i - x_i^\prime \ul{\beta})^2/n}\right)^{1/2},\\
& \ul{A}_{\sigma xl} = \ul{c}_\sigma/\left(\sum_i(x_{il} - \sum_i x_{il}/n)^2/n\right), \; 
\ul{B}_{\sigma xl} = \ul{c}_\sigma/\left({\sum_i(x_{il} - \sum_i x_{il}/n)^2/n}\right)^{1/2},\\
& \ul{A}_{s xk} = \ul{c}_s, \; \ul{B}_{s xk} = \ul{c}_s,\; \ul{A}_{s y}=\ul{c}_s, \; \ul{B}_{s y}=\ul{c}_s,\\
& \ul{a}=15, \;\ul{A}_m=1,
\end{align*}
where $\ul{c}_\beta=100$, $\ul{c}_\sigma=0.1$, $\ul{c}_s=10$.  
Thus, a modal prior draw would have one mixture component and it would be near a normal linear regression estimated by the least squares.
As Figure \ref{fig:prior} illustrates, the prior variances are chosen sufficiently large so that  a wide range of densities can be easily accommodated by the prior.
\begin{figure}
\includegraphics[width=100mm]{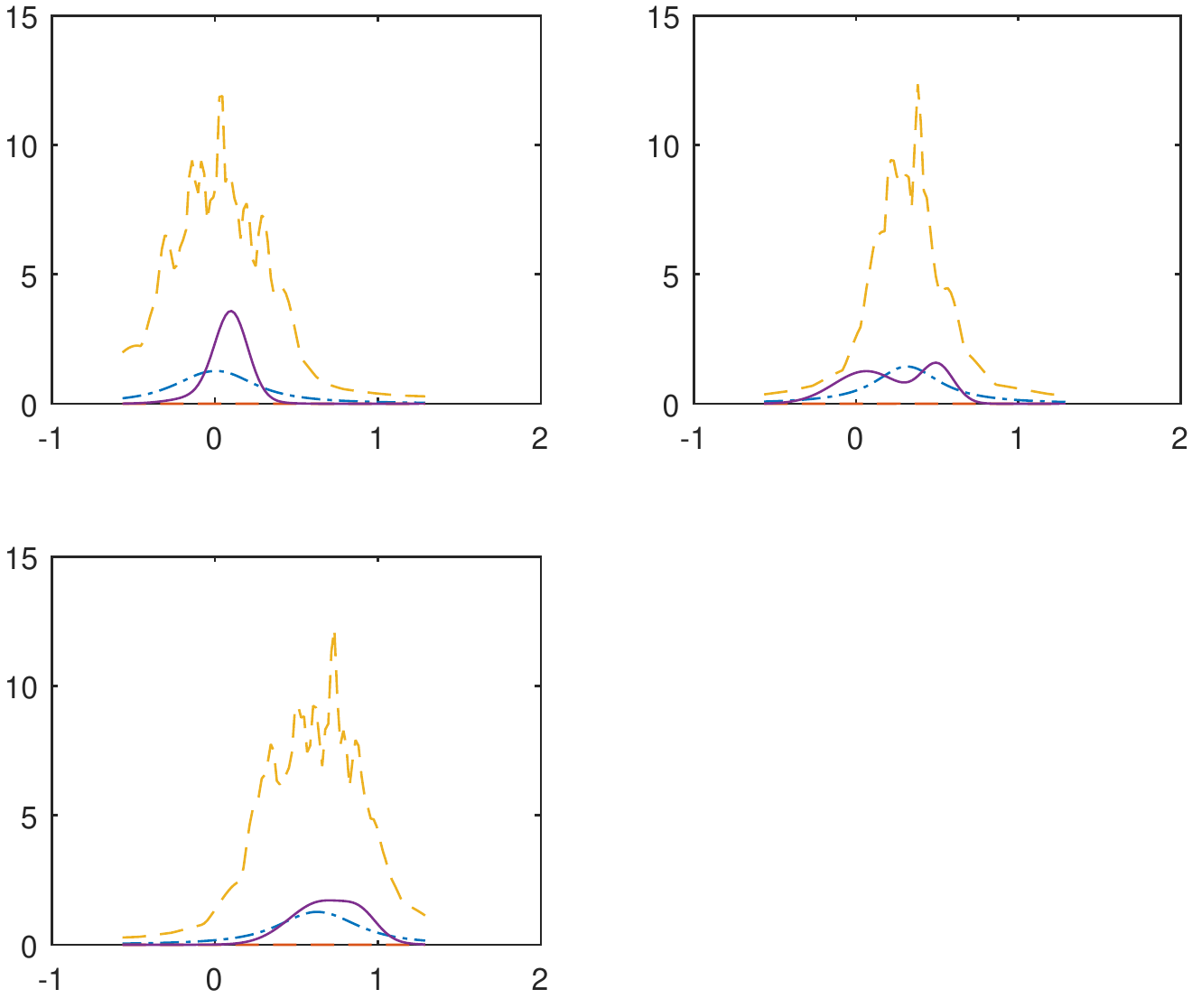}%
\caption{
Simulated prior conditional densities for $d_x=1$ and $x \in \{0.1,0.5,0.9\}$: the solid lines are the DGP values,
the dash-dotted lines are the prior means, and the dashed lines are pointwise 99.99\% 
 prior credible intervals.
}
\label{fig:prior}
\end{figure}

The DGP for simulation experiments is as follows: $x_i=(x_{i1},\ldots,x_{id_x})$, $x_{ik} \sim U[0,1]$ (or $x_{ik} \sim N(0.5, 12^{-1/2})$ for an unbounded support case) and the true conditional density is 
\begin{align}
\label{eq:DGP}
f_0(y_i|x_{i1})=e^{-2x_{i1}}N(y_i;x_{i1},0.1^2)+(1-e^{-2x_{i1}})N(y_i;x_{i1}^4,0.2^2).
\end{align}
Note that the DGP conditional density depends only on the first coordinate of $x_i$, the rest of the coordinates are irrelevant.  
This DGP was previously used without irrelevant covariates by 
\cite{DunsonPillaiPark:07}, \cite{DunsonPark:08}, and \cite{NoretsPelenis:11}.

For each estimation exercise, we
perform $5,000$ MCMC iterations, of which the first $500$ are discarded for burn-in.
The MCMC draws of $m$ mostly belong to $\{3,\ldots,13\}$.
Figure \ref{fig:fig100dx2} presents Bayesian and kernel estimation results for one dataset of size $n=1000$ and $d_x=2$.
Each panel in the figure shows the DGP densities, the kernel estimates, the posterior means, and the posterior 0.01\%-quantiles
conditional on a particular value of covariate $x \in \{0.1,0.5,0.9\}^2$. 
As can be seen from the figure, the estimation results from both approaches can be pretty close.
\begin{figure}
\includegraphics[width=160mm]{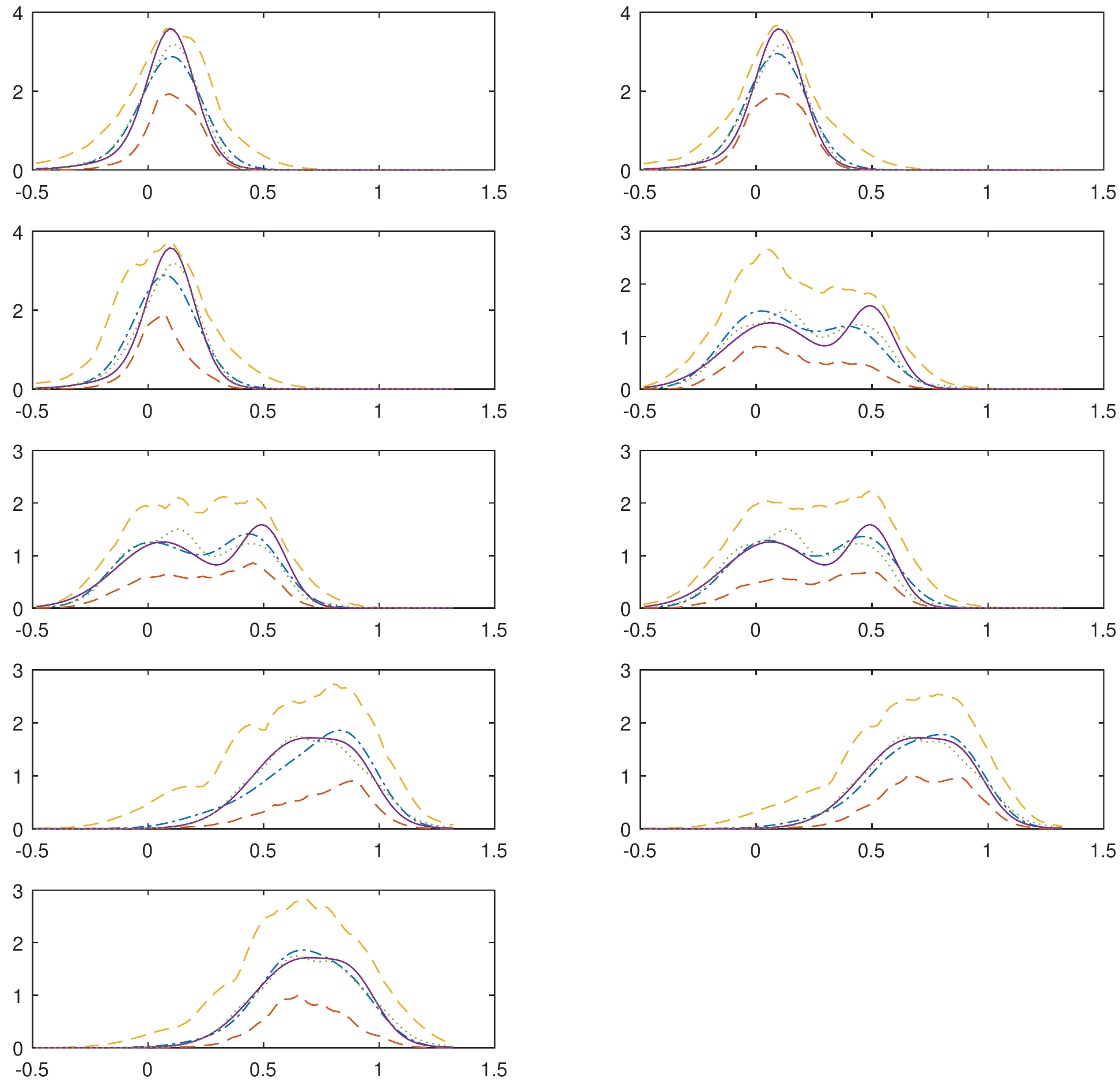}%
\caption{
Estimated conditional densities for $d_x=2$ and $x \in \{0.1,0.5,0.9\}^2$.  The solid lines are the true values,
the dash-dotted lines are the posterior means, the dotted lines are kernel estimates, and the dashed lines are pointwise 99.99\% 
 credible intervals.
}
\label{fig:fig100dx2}
\end{figure}

In every Monte Carlo experiment we perform, 50 simulated datasets are used.
For each dataset, the performance of an estimator is evaluated by the mean absolute error 
\begin{align*}
%RMSE &= \sqrt{ \frac{\sum_{i=1}^{N_y}\sum_{j=1}^{N_x}\left( \hat{f}(y_i|x_j)-f_0(y_i|x_j) \right)^2}{N_yN_x}  }\\
\mbox{MAE} & = \frac{\sum_{i=1}^{N_y}\sum_{j=1}^{N_x}\left| \hat{f}(y_i|x_j)-f_0(y_i|x_j) \right|}{N_yN_x},  
\end{align*}
where $x_j\in\{0.1,0.5,0.9\}^{d_x}$ and $y_i$ belongs to a 100 points equal spaced grid on the range of simulated values for $y$.
The results reported below are qualitatively the same for the root mean squared error.

Table \ref{tb:MC} presents estimation results for 9 
Monte Carlo experiments based on different values of $n$, $d_x$, the covariate support, and the prior.
The table gives MAE for the kernel and posterior mean estimators averaged over 50 simulated datasets.
It also shows the average difference between MAEs of the two estimators and the corresponding $t$-statistics.
In all the experiments, the posterior mean estimator performs better than the kernel estimator and the differences are highly statistically significant. 
The first three rows of the table present the results for the covariates with bounded support, $d_x=1$, and $n \in \{10^2,10^3,10^4\}$.  As expected, the MAE decreases as the sample size increases for both estimators.
The next five rows of the table show the results for $d_x \in \{1,3,5\}$ and $n =10^3$ for covariates with bounded and unbounded support.  
Even though the posterior mean outperforms the kernel estimator in absolute terms, the MAE for kernel estimator impressively changes little when the number of irrelevant covariates increases.
The last raw of the table shows the results for $d_x=1$, $n=10^3$, and the following alternative prior hyperparameters:
$\ul{c}_\beta=200$, $\ul{c}_\sigma=0.2$, $\ul{c}_s=15$, $\ul{a}=12$, and $\ul{A}_m=2$.
Thus, the results are not very sensitive to reasonable variations in $(\ul{c}_\beta, \ul{c}_\sigma,\ul{c}_s,\ul{a},\ul{A}_m)$.

The dimension of the covariates does not noticeably affect the computing time for the posterior mean estimator.
The computations for the kernel estimator are very fast for low-dimensional covariates.
They slow down considerably when $d_x$ increases.
For $d_x=5$, the posterior mean is slightly faster to compute than the kernel estimator.

Overall, the Monte Carlo experiments suggest that the model proposed in this paper is a practical and promising alternative to 
classical non-parametric methods.

\begin{table}[htbp]
 \centering
  \caption{MAE for kernel and posterior mean estimators}
	\label{tb:MC}
    \begin{tabular}{cccccccc}
    \toprule
$x_{ik} \sim g_0$ 	& $d_x$ & n 			 & Bayes & Kernel & B-K   & \%(B $<$ K) & t-stat \\
    \midrule
$U[0,1]$ 						&  1    & $10^2$   & 0.107 & 0.164 & -0.058 & 1  					& -15.47 \\
$U[0,1]$ 						&  1    & $10^4$   & 0.032 & 0.040 & -0.008 & 0.88  			& -8.28 \\
$U[0,1]$ 						&  1    & $10^3$   & 0.062 & 0.096 & -0.033 & 1  					& -16.16 \\
$U[0,1]$ 						&  3    & $10^3$   & 0.074 & 0.097 & -0.022 & 0.96  			& -13.40 \\
$U[0,1]$ 						&  5    & $10^3$   & 0.084 & 0.098 & -0.015 & 0.86  			& -7.88 \\
$N(0.5, 12^{-1/2})$ & 1     & $10^3$   & 0.028 & 0.054 & -0.026 & 1  					& -15.33 \\
$N(0.5, 12^{-1/2})$ & 3     & $10^3$   & 0.033 & 0.054 & -0.021 & 1  					& -11.78 \\
$N(0.5, 12^{-1/2})$ & 5     & $10^3$   & 0.038 & 0.054 & -0.017 & 0.92 			 	& -8.88 \\
$U[0,1]$ 						&  1    & $10^3$   & 0.060 & 0.096 & -0.036 & 1  					& -17.72 \\
				\bottomrule
    \end{tabular}%
 \end{table}%

\section{Conclusion}
\label{sec:conclusion}

We show above that under a reasonable prior distribution, the posterior contraction rate in our model is bounded above by
$\epsilon_n=n^{-\beta/(2\beta+d)} (\log n)^t$
for any
\[
t > [d(1 + 1/\beta + 1/\tau) + \max\{\tau_1,1,\tau_2/\tau\}] / (2 + d/\beta) +  \max \{0, (1- \tau_1)/2\}.
\]
Rate $n^{-\beta/(2\beta+d)}$ is minimax for estimation of multivariate densities
when their smoothness level is $\beta$ and dimension of $(y,x)$ is $d$.
Since the total variation distance between joint densities  for $(y,x)$ is bounded by the sum of the integrated total variation distance between the conditional densities and the total variation distance between
the densities of $x$, the minimax rate for estimation of conditional densities of smoothness $\beta$ 
in integrated total variation distance
cannot be faster than $n^{-\beta/(2\beta+d)}$.  Thus, we can claim that our Bayesian nonparametric model achieves optimal contraction rate up to a log factor.
We are not aware of analogous results for estimators based on kernels or mixtures. 
In the classical settings, \cite{Efromovich:07} develops an estimator based on orthogonal series that achieves minimax rates for one-dimensional $y$ and $x$.
In a recent paper, \cite{shen2014adaptive} consider a compactly supported Bayesian model for conditional densities based on tensor products of spline functions.  
They show that under suitable sparsity assumptions, 
the posterior contracts at an optimal rate even when 
the dimension of covariates increases exponentially with the sample size.
An advantage of our results is that we do not need to assume a known upper bound on the smoothness level and 
the boundedness away from zero for the true density.
The analysis of the posterior contraction rates in our model under 
sparsity and increasing dimension of covariates is an important direction for future work.

\bibliographystyle{ecta}
\bibliography{allreferences}

\begin{thebibliography}{40}
\newcommand{\enquote}[1]{``#1''}
\expandafter\ifx\csname natexlab\endcsname\relax\def\natexlab#1{#1}\fi

\bibitem[\protect\citeauthoryear{Barron, Schervish, and Wasserman}{Barron
  et~al.}{1999}]{BarronSchervishWasserman:99}
\textsc{Barron, A., M.~J. Schervish, and L.~Wasserman} (1999): \enquote{The
  Consistency of Posterior Distributions in Nonparametric Problems,} \emph{The
  Annals of Statistics}, 27, 536--561.

\bibitem[\protect\citeauthoryear{Bhattacharya, Pati, Dunson
  et~al.}{Bhattacharya et~al.}{2014}]{bhattacharya2014anisotropic}
\textsc{Bhattacharya, A., D.~Pati, D.~Dunson, et~al.} (2014):
  \enquote{Anisotropic function estimation using multi-bandwidth Gaussian
  processes,} \emph{The Annals of Statistics}, 42, 352--381.

\bibitem[\protect\citeauthoryear{Chung and Dunson}{Chung and
  Dunson}{2009}]{ChungDunson:09}
\textsc{Chung, Y. and D.~B. Dunson} (2009): \enquote{Nonparametric Bayes
  Conditional Distribution Modeling With Variable Selection,} \emph{Journal of
  the American Statistical Association}, 104, 1646--1660.

\bibitem[\protect\citeauthoryear{De~Iorio, Muller, Rosner, and
  MacEachern}{De~Iorio et~al.}{2004}]{DeIorioMullerRosnerMacEachern:04}
\textsc{De~Iorio, M., P.~Muller, G.~L. Rosner, and S.~N. MacEachern} (2004):
  \enquote{An ANOVA Model for Dependent Random Measures,} \emph{Journal of the
  American Statistical Association}, 99, 205--215.

\bibitem[\protect\citeauthoryear{de~Jonge and van Zanten}{de~Jonge and van
  Zanten}{2010}]{deJongeVanZanten2010}
\textsc{de~Jonge, R. and J.~H. van Zanten} (2010): \enquote{Adaptive
  nonparametric Bayesian inference using location-scale mixture priors,}
  \emph{The Annals of Statistics}, 38, 3300--3320.

\bibitem[\protect\citeauthoryear{Dunson and Park}{Dunson and
  Park}{2008}]{DunsonPark:08}
\textsc{Dunson, D.~B. and J.-H. Park} (2008): \enquote{Kernel stick-breaking
  processes,} \emph{Biometrika}, 95, 307--323.

\bibitem[\protect\citeauthoryear{Dunson, Pillai, and Park}{Dunson
  et~al.}{2007}]{DunsonPillaiPark:07}
\textsc{Dunson, D.~B., N.~Pillai, and J.-H. Park} (2007): \enquote{Bayesian
  Density Regression,} \emph{Journal of the Royal Statistical Society. Series B
  (Statistical Methodology)}, 69, pp. 163--183.

\bibitem[\protect\citeauthoryear{Efromovich}{Efromovich}{2007}]{Efromovich:07}
\textsc{Efromovich, S.} (2007): \enquote{Conditional density estimation in a
  regression setting,} \emph{The Annals of Statistics}, 35, 2504--2535.

\bibitem[\protect\citeauthoryear{Geweke and Keane}{Geweke and
  Keane}{2007}]{Geweke:07}
\textsc{Geweke, J. and M.~Keane} (2007): \enquote{Smoothly mixing regressions,}
  \emph{Journal of Econometrics}, 138, 252--290.

\bibitem[\protect\citeauthoryear{Ghosal, Ghosh, and Ramamoorthi}{Ghosal
  et~al.}{1999}]{GhosalGhoshRamamoorthi:99}
\textsc{Ghosal, S., J.~K. Ghosh, and R.~V. Ramamoorthi} (1999):
  \enquote{Posterior Consistency of Dirichlet Mixtures in Density Estimation,}
  \emph{The Annals of Statistics}, 27, 143--158.

\bibitem[\protect\citeauthoryear{Ghosal, Ghosh, and Vaart}{Ghosal
  et~al.}{2000}]{GhosalGhoshVaart:2000}
\textsc{Ghosal, S., J.~K. Ghosh, and A.~W. v.~d. Vaart} (2000):
  \enquote{Convergence Rates of Posterior Distributions,} \emph{The Annals of
  Statistics}, 28, 500--531.

\bibitem[\protect\citeauthoryear{Ghosal and van~der Vaart}{Ghosal and van~der
  Vaart}{2007}]{GhosalVandervaart:07}
\textsc{Ghosal, S. and A.~van~der Vaart} (2007): \enquote{Posterior convergence
  rates of {D}irichlet mixtures at smooth densities,} \emph{The Annals of
  Statistics}, 35, 697--723.

\bibitem[\protect\citeauthoryear{Ghosal and van~der Vaart}{Ghosal and van~der
  Vaart}{2001}]{GhosalVaart:01}
\textsc{Ghosal, S. and A.~W. van~der Vaart} (2001): \enquote{Entropies and
  rates of convergence for maximum likelihood and Bayes estimation for mixtures
  of normal densities,} \emph{The Annals of Statistics}, 29, 1233--1263.

\bibitem[\protect\citeauthoryear{Griffin and Steel}{Griffin and
  Steel}{2006}]{GriffinSteel:06}
\textsc{Griffin, J.~E. and M.~F.~J. Steel} (2006): \enquote{Order-Based
  Dependent Dirichlet Processes,} \emph{Journal of the American Statistical
  Association}, 101, 179--194.

\bibitem[\protect\citeauthoryear{Hall, Racine, and Li}{Hall
  et~al.}{2004}]{Hall2004}
\textsc{Hall, P., J.~Racine, and Q.~Li} (2004): \enquote{Cross-Validation and
  the Estimation of Conditional Probability Densities,} \emph{Journal of the
  American Statistical Association}, 99, 1015--1026.

\bibitem[\protect\citeauthoryear{Hayfield and Racine}{Hayfield and
  Racine}{2008}]{Hayfield2008}
\textsc{Hayfield, T. and J.~S. Racine} (2008): \enquote{Nonparametric
  Econometrics: The np Package,} \emph{Journal of Statistical Software}, 27,
  1--32.

\bibitem[\protect\citeauthoryear{Huang}{Huang}{2004}]{Huang:04}
\textsc{Huang, T.-M.} (2004): \enquote{Convergence rates for posterior
  distributions and adaptive estimation,} \emph{Ann. Statist.}, 32, 1556--1593.

\bibitem[\protect\citeauthoryear{Jacobs, Jordan, Nowlan, and Hinton}{Jacobs
  et~al.}{1991}]{JacobsEtAl:91}
\textsc{Jacobs, R.~A., M.~I. Jordan, S.~J. Nowlan, and G.~E. Hinton} (1991):
  \enquote{Adaptive mixtures of local experts,} \emph{Neural Computation}, 3,
  79--87.

\bibitem[\protect\citeauthoryear{Jordan and Xu}{Jordan and
  Xu}{1995}]{JordanXu:95}
\textsc{Jordan, M. and L.~Xu} (1995): \enquote{Convergence results for the EM
  approach to mixtures of experts architectures,} \emph{Neural Networks}, 8,
  1409 -- 1431.

\bibitem[\protect\citeauthoryear{Keane and Stavrunova}{Keane and
  Stavrunova}{2011}]{KeaneStavrunova2011SMRApp}
\textsc{Keane, M. and O.~Stavrunova} (2011): \enquote{A smooth mixture of
  Tobits model for healthcare expenditure,} \emph{Health Economics}, 20,
  1126--1153.

\bibitem[\protect\citeauthoryear{Kruijer, Rousseau, and van~der Vaart}{Kruijer
  et~al.}{2010}]{KruijerRousseauVaart:09}
\textsc{Kruijer, W., J.~Rousseau, and A.~van~der Vaart} (2010):
  \enquote{Adaptive {B}ayesian density estimation with location-scale
  mixtures,} \emph{Electronic Journal of Statistics}, 4, 1225--1257.

\bibitem[\protect\citeauthoryear{Li, Villani, and Kohn}{Li
  et~al.}{2010}]{LiVillaniKohn2010}
\textsc{Li, F., M.~Villani, and R.~Kohn} (2010): \enquote{Flexible modeling of
  conditional distributions using smooth mixtures of asymmetric student
  \textit{t} densities,} \emph{Journal of Statistical Planning and Inference},
  140, 3638--3654.

\bibitem[\protect\citeauthoryear{Li and Racine}{Li and
  Racine}{2007}]{LiRacine2007}
\textsc{Li, Q. and J.~S. Racine} (2007): \emph{Nonparametric Econometrics:
  Theory and Practice}, Princeton University Press.

\bibitem[\protect\citeauthoryear{MacEachern}{MacEachern}{1999}]{MacEachern:99}
\textsc{MacEachern, S.~N.} (1999): \enquote{Dependent Nonparametric Processes,}
  \emph{ASA Proceedings of the Section on Bayesian Statistical Science}.

\bibitem[\protect\citeauthoryear{Norets}{Norets}{2010}]{Norets_aos:10}
\textsc{Norets, A.} (2010): \enquote{Approximation of conditional densities by
  smooth mixtures of regressions,} \emph{The Annals of Statistics}, 38,
  1733--1766.

\bibitem[\protect\citeauthoryear{Norets and Pelenis}{Norets and
  Pelenis}{2012}]{NoretsPelenis2012}
\textsc{Norets, A. and J.~Pelenis} (2012): \enquote{Bayesian modeling of joint
  and conditional distributions,} \emph{Journal of Econometrics}, 168,
  332--346.

\bibitem[\protect\citeauthoryear{Norets and Pelenis}{Norets and
  Pelenis}{2014}]{NoretsPelenis:11}
---\hspace{-.1pt}---\hspace{-.1pt}--- (2014): \enquote{Posterior Consistency in
  Conditional Density Estimation by Covariate Dependent Mixtures,} Econometric
  Theory.

\bibitem[\protect\citeauthoryear{Pati, Dunson, and Tokdar}{Pati
  et~al.}{2013}]{pati2013}
\textsc{Pati, D., D.~B. Dunson, and S.~T. Tokdar} (2013): \enquote{Posterior
  consistency in conditional distribution estimation,} \emph{Journal of
  Multivariate Analysis}, 116, 456--472.

\bibitem[\protect\citeauthoryear{Peng, Jacobs, and Tanner}{Peng
  et~al.}{1996}]{PengJacobsTanner:1996}
\textsc{Peng, F., R.~A. Jacobs, and M.~A. Tanner} (1996): \enquote{Bayesian
  Inference in Mixtures-of-Experts and Hierarchical Mixtures-of-Experts Models
  With an Application to Speech Recognition,} \emph{Journal of the American
  Statistical Association}, 91, 953--960.

\bibitem[\protect\citeauthoryear{Rousseau}{Rousseau}{2010}]{Rousseau:10}
\textsc{Rousseau, J.} (2010): \enquote{Rates of convergence for the posterior
  distributions of mixtures of betas and adaptive nonparametric estimation of
  the density,} \emph{The Annals of Statistics}, 38, 146--180.

\bibitem[\protect\citeauthoryear{Scricciolo}{Scricciolo}{2006}]{Scricciolo:06}
\textsc{Scricciolo, C.} (2006): \enquote{{Convergence rates for Bayesian
  density estimation of infinite-dimensional exponential families.}}
  \emph{Annals of Statatistics}, 34, 2897--2920.

\bibitem[\protect\citeauthoryear{Shen and Ghosal}{Shen and
  Ghosal}{2014}]{shen2014adaptive}
\textsc{Shen, W. and S.~Ghosal} (2014): \enquote{Adaptive Bayesian density
  regression for high-dimensional data,} \emph{arXiv preprint arXiv:1403.2695}.

\bibitem[\protect\citeauthoryear{Shen, Tokdar, and Ghosal}{Shen
  et~al.}{2013}]{ShenTokdarGhosal2013}
\textsc{Shen, W., S.~T. Tokdar, and S.~Ghosal} (2013): \enquote{Adaptive
  Bayesian multivariate density estimation with Dirichlet mixtures,}
  \emph{Biometrika}, 100, 623--640.

\bibitem[\protect\citeauthoryear{Tokdar, Zhu, and Ghosh}{Tokdar
  et~al.}{2010}]{Tokdar2010}
\textsc{Tokdar, S., Y.~Zhu, and J.~Ghosh} (2010): \enquote{Bayesian density
  regression with logistic Gaussian process and subspace projection,}
  \emph{Bayesian Analysis}, 5, 319--344.

\bibitem[\protect\citeauthoryear{van~der Vaart and van Zanten}{van~der Vaart
  and van Zanten}{2009}]{VaartZanten:09}
\textsc{van~der Vaart, A.~W. and J.~H. van Zanten} (2009): \enquote{Adaptive
  {B}ayesian estimation using a {G}aussian random field with inverse gamma
  bandwidth,} \emph{The Annals of Statistics}, 37, 2655--2675.

\bibitem[\protect\citeauthoryear{Villani, Kohn, and Giordani}{Villani
  et~al.}{2009}]{VillaniKohnGiordani:07}
\textsc{Villani, M., R.~Kohn, and P.~Giordani} (2009): \enquote{Regression
  density estimation using smooth adaptive Gaussian mixtures,} \emph{Journal of
  Econometrics}, 153, 155 -- 173.

\bibitem[\protect\citeauthoryear{Villani, Kohn, and Nott}{Villani
  et~al.}{2012}]{VillaniKohnNott2012}
\textsc{Villani, M., R.~Kohn, and D.~J. Nott} (2012): \enquote{Generalized
  smooth finite mixtures,} \emph{Journal of Econometrics}, 171, 121 -- 133.

\bibitem[\protect\citeauthoryear{Wade, Dunson, Petrone, and Trippa}{Wade
  et~al.}{2014}]{wade2014improving}
\textsc{Wade, S., D.~B. Dunson, S.~Petrone, and L.~Trippa} (2014):
  \enquote{Improving prediction from Dirichlet process mixtures via
  enrichment,} \emph{The Journal of Machine Learning Research}, 15, 1041--1071.

\bibitem[\protect\citeauthoryear{Wood, Jiang, and Tanner}{Wood
  et~al.}{2002}]{WoodJiangTanner:02}
\textsc{Wood, S., W.~Jiang, and M.~Tanner} (2002): \enquote{Bayesian mixture of
  splines for spatially adaptive nonparametric regression,} \emph{Biometrika},
  89, 513--528.

\bibitem[\protect\citeauthoryear{Yang and Tokdar}{Yang and
  Tokdar}{2014}]{yang2014minimax}
\textsc{Yang, Y. and S.~T. Tokdar} (2014): \enquote{Minimax risks for high
  dimensional nonparametric regression,} \emph{arXiv preprint arXiv:1401.7278}.

\end{thebibliography}

\section{Appendix}

\begin{lemma} 
\label{lm:dH_cond_bdd_dHjoint}
Suppose $f, f_0 \in \mathcal{F}$, $g_0(x) \leq \bar{g} < \infty$, $g(x)$ and $u(x)$ are densities on $\mathcal{X}$, 
$u(x) \geq \underline{u}>0$.  Then,
\[
d_h^2(f_0,f) \leq \frac{4\bar{g}}{\underline{u}} \int \left(\sqrt{f_0(y|x) u(x)}-\sqrt{f(y|x) g(x)}\right)^2 dydx.
\]
\end{lemma}

\begin{proof} %(Lemma \ref{lm:dH_cond_bdd_dHjoint})
Observe that 
\begin{eqnarray}
d_h^2(f_0,f) &=& \int \left(\sqrt{f_0(y | x)} - \sqrt{f(y | x)}\right)^2  g_0(x) dydx \nonumber\\
 &\leq&  \frac{\bar{g}}{\underline{u}} \int \left(\sqrt{f_0(y | x) u(x)} - \sqrt{f(y | x)u(x)}\right)^2dydx \nonumber\\
 &\leq&  \frac{2\bar{g}}{\underline{u}} (\mbox{I} + \mbox{II}), \label{eq0:lem} 
\label{eq:ineq_I_II}
\end{eqnarray}
where $\mbox{I} =\int \left(\sqrt{f_0(y | x) u(x)} - \sqrt{f(y | x)g(x)}\right)^2dydx$, $\mbox{II}=  \int \left(\sqrt{f(y | x) g(x)} - \sqrt{f(y | x)u(x)}\right)^2dydx$.  
 
Observe that
\begin{eqnarray}
\mbox{II} \leq  \int \left(\sqrt{g(x)} - \sqrt{u(x)}\right)^2dx = 2 \left( 1- \int \sqrt{g(x) u(x)} dx \right)  \leq  \mbox{I}.\label{eq1:lem}
  \end{eqnarray}
The final inequality in \eqref{eq1:lem} follows since $\int \sqrt{f_0(y |x) f(y|x)} dy \leq \frac{1}{2} \left(\int  f_0(y |x)dy + \int f(y|x)dy  \right) =1$. 
Combining \eqref{eq0:lem} and \eqref{eq1:lem}, we obtain 
\begin{eqnarray*}
d_h^2(f_0,f) \leq   4\mbox{I} = \frac{4\bar{g}}{\underline{u}}\int \left(\sqrt{f_0(y|x) u(x)}-\sqrt{f(y|x) g(x)} \right)^2 dydx. 
\end{eqnarray*}

\end{proof}

\begin{corollary} 
\label{cor:dH_cond_bdd_dHjoint}
Suppose $f, f_0 \in \mathcal{F}$,  $g(x)$ and $\bar{g}_0(x)$ are densities on $\mathcal{X}$,  with $\bar{g}_0$ satisfying $\eta \bar{g}_0(x) \geq g_0(x)$ for some constant $\eta > 0$ and all $x \in \mathcal{X}$.  Then,
\[
d_h^2(f_0,f) \leq 4 \eta \int \left(\sqrt{f_0(y|x) \bar{g}_0(x)}-\sqrt{f(y|x) g(x)}\right)^2 dydx.
\]
\end{corollary}
To prove the corollary note that the inequality \eqref{eq:ineq_I_II} in the proof of Lemma \ref{lm:dH_cond_bdd_dHjoint} holds under $\eta \bar{g}_0(x) \geq g_0(x)$ with 
$u$ replaced by $\bar{g}_0$ and $\bar{g}/\underline{u}$ replaced by $\eta$.  The rest of the lemma's proof applies with $\bar{g}_0$ replacing $u$.

\begin{lemma} 
\label{lm:positive_mix_dens}

In Theorem 3 of \cite{ShenTokdarGhosal2013},
replace their
$g_\sigma=f_{\sigma} + (1/2) f_0 1\{f_{\sigma}  < (1/2) f_0 \}$
with $g_{\sigma} = f_{\sigma} + 2 |f_{\sigma}| 1 \{f_{\sigma}  < 0 \}$, where notation from \cite{ShenTokdarGhosal2013} is used.  Then, the claim of the theorem holds.

\end{lemma}

\begin{proof}

With the alternative definition of $g_\sigma$, the proof 
of \cite{ShenTokdarGhosal2013} goes through with the following changes.
First,
$1 \leq \int g_{\sigma}(x) dx = \int f_{\sigma}(x) dx  + 2 \int |f_{\sigma}| 1\{(f_{\sigma} < 0\} \leq 1+ 3 \int_{A_{\sigma}^c} f_0(x)dx \leq 1+ K_2 \sigma^{2\beta}$.
Second, replace inequality $r_\sigma \leq g_\sigma$   with   $(1/2) r_\sigma \leq g_\sigma$.

\end{proof}

\begin{lemma}
\label{lm:dH_KL}
There is a $\lambda_0 \in (0,1)$ such that for any $\lambda \in (0,\lambda_0)$ and any two conditional densities $p,q \in \mathcal{F}$, a probability measure $P$ on $\mathcal{Z}$ that has a conditional density equal to $p$, and $d_h$ defined with the distribution on $\mathcal{X}$ implied by $P$,
\[
P \log \frac{p}{q} \leq d_h^2(p,q) \left( 1+ 2 \log\frac{1}{\lambda}\right) 
+ 2 P \left \{  \left(\log \frac{p}{q} \right) 1 \left( \frac{q}{p}\leq \lambda \right) \right\},
\]
\[
P \left(\log \frac{p}{q} \right)^2 \leq d_h^2(p,q) \left( 12+ 2 \left(\log\frac{1}{\lambda}\right)^2\right) 
+ 8 P \left \{  \left(\log \frac{p}{q} \right)^2 1 \left( \frac{q}{p}\leq \lambda \right) \right\},
\]
\end{lemma}

\begin{proof}
The proof is exactly the same as the proof of Lemma 4 of \cite{ShenTokdarGhosal2013}, which in turn, follows the proof of Lemma 7 in \cite{GhosalVandervaart:07}.
\end{proof}

%\section{Posterior contraction result for the method in Section \ref{sec:simul}}

\begin{theorem}
\label{thm:lincoeflocals} 
Assume $f_0$ satisfies the assumptions in Section \ref{sec:irr_cov} with $d_y =1$.  Then the model \eqref{eq:cond_mix_finsam} in Section \ref{sec:simul} and the prior specifications following it leads to the same posterior contraction rate as specified in Corollary \ref{cr:irr_cov}.
\end{theorem}
\begin{proof}
In the following we will verify prior 
thickness condition with the same $\tilde{\epsilon}_n$ (with $d_y =1$) as in Corollary \ref{cr:irr_cov} and modify the sieve construction accordingly. 
The proof proceeds along the lines of the proof of Corollary \ref{cr:irr_cov}.  The main difference is that the following joint density is used in bounds for the 
distance between conditional densities
\begin{align*}
\tilde{p}(y, x \mid \theta, m) = &
\sum_{j=1}^m \frac{\alpha_j \exp\{-0.5\sum_{k=1}^{d_x}(x_k - \mu_{jk}^x)^2
/(\sigma^x_k s^x_{jk})^2\}}{\sum_{i=1}^m \alpha_i \exp\{-0.5\sum_{k=1}^{d_x}(x_k - \mu_{ik}^x)^2
/(\sigma^x_k s^x_{ik})^2\}}  \phi_{\mu_j^y+ x'\beta_j, \sigma^y s^y_j}(y) 
\\
& \cdot \sum_{j=1}^m \alpha_j \phi_{\mu_{j}^x, \sigma^x \circ s^x_j}(x),
\end{align*} 
where  $\circ$ denotes the Hadamard product.  The  intercept absorbed in the notation ``$x'\beta_j$" in \eqref{eq:cond_mix_finsam} is denoted by $\mu_j^y$ here.   
Let
\begin{align*}
\theta_{d_x^0}=
\{ &
\mu_j^y, \mu_{j1d_x^0}^{x}=(\mu_{j1}^{x}, \ldots,\mu_{j d_x^0}^{x}), \alpha_j, 
s^y_j, s^x_{j1d_x^0}=(s^x_{j1},\ldots,s^x_{jd_x^0}), \\
& \beta_{j1d_x^0}=(\beta_{j1},\ldots,\beta_{jd_x^0}), \; j=1,2,\ldots; \; \sigma^y, \sigma^x_{1d_x^0}=(\sigma^x_1,\ldots,\sigma^x_{d_x^0})
\},
\end{align*}
\begin{align*}
	S_{\theta^\star}=&\big \{
	(\mu_j, \alpha_j, \, j=1,2,\ldots; \sigma^y, \sigma^x): 
	 \; (\mu_j^y,\mu_{j1}^{x}, \ldots,\mu_{j d_x^0}^{x})  \in U_j,\; \\
	&
	||(\mu_{j d_x^0+1}^{x}, \ldots,\mu_{j d_x}^{x})|| \leq \sigma_n \tilde{\epsilon}_n^{2b_1},  ||\beta_j|| \leq \sigma_n \tilde{\epsilon}_n^{2b_1}\;
	j\leq K; \\
	& \sum_{j=1}^K \abs{\alpha_j - \alpha_j^\star} \leq 2\tilde{\epsilon}_n^{2 d^0 b_1}, \, \min_{j=1, \ldots, K} \alpha_j \geq \tilde{\epsilon}_n^{4 d^0 b_1}/2; 
	\\	& 
	(\sigma^x_k)^2, (\sigma^y)^2 \in [ \sigma_n^{2}/(1+ \sigma_n^{2\beta}), \sigma_n^{2} ],\;  k \leq d_x^0;\\
	& (\sigma^x_k)^2 \in [ a_{\sigma_n}^2,  2a_{\sigma_n}^2],\; 	k=d_x^0+1, \ldots, d_x
	; \\
	&  s^x_{jk}, s^y_j \in  [1,1+\sigma_n^{2\beta}], j=1,2,\ldots, K;  k=1, \ldots, d_x \big \},
\end{align*}
and $s^x_{jk}=s^y_j=1$ and $\beta_{jk}=0$ in $\theta_{d_x^0}^\ast$  for $k=1, 2, \ldots, d_x^0$ and $j=1,\ldots,K$.

Similarly to the proof of Corollary \ref{cr:irr_cov},
\begin{align*}
d_h(f_0, p (\cdot | \cdot, \theta,m)) \lesssim  
\sigma_n^\beta  
 + d_H(\tilde{p}(\cdot |  \theta^\ast_{d_x^0},m), \tilde{p}(\cdot |  \theta_{d_x^0},m)) 
 + d_h(p(\cdot | \cdot, \theta_{d_x^0},m), p(\cdot | \cdot, \theta,m)).
\end{align*}
Consider $\theta \in S_{\theta^\star}$ and let $s_{\cdot j}=\prod_{k=1}^{d_x^0}s_{jk}$.
Then, $d_H(\tilde{p}(\cdot |  \theta^\ast_{d_x^0},m), \tilde{p}(\cdot |  \theta_{d_x^0},m))^2$ can be bounded by 
\begin{align*}
 & \norm
{ 
	\sum_{j=1}^K \alpha_j^* \phi_{\mu_j^*, \sigma_n}(\cdot) - 
	\sum_{j=1}^K \alpha_j s_{\cdot j} \phi_{\mu_j^y + x_{1d_x^0}^\prime \beta_{j1d_x^0}, s^y_j \sigma_j^y}(\cdot) 
																		\phi_{\mu_{j1d_x^0}^x,s^x_{j1d_x^0} \circ \sigma^x_{1d_x^0}}(\cdot) 
	\frac{\sum_{j=1}^K \alpha_j \phi_{\mu_{j1d_x^0}^x, s^x_{j1d_x^0} \circ \sigma^x_{1d_x^0}}(\cdot)}
	  	 {\sum_{j=1}^K \alpha_j s_{\cdot j} \phi_{\mu_{j1d_x^0}^x, s^x_{j1d_x^0} \circ \sigma^x_{1d_x^0}}(\cdot)} 
}_1\\
 & \leq 
\norm
{
	\sum_{j=1}^K \alpha_j^* \phi_{\mu_j^*, \sigma_n}(\cdot) - \sum_{j=1}^K \alpha_j s_{\cdot j}
													\phi_{\mu_j^y + x_{1d_x^0}^\prime \beta_{j1d_x^0}, s^y_j \sigma^y}(\cdot) 
													\phi_{\mu_{j1d_x^0}^x, s^x_{j1d_x^0}\circ \sigma^x_{1d_x^0}}(\cdot)
}_1 +
 \\  
& \;\;\; \; 
\norm
{ 
	\sum_{j=1}^K \alpha_j s_{\cdot j} \phi_{\mu_j^y + x_{1d_x^0}^\prime \beta_{j1d_x^0}, s^y_j \sigma^y}(\cdot) 
																		\phi_{\mu_{j1d_x^0}^x, s^x_{j1d_x^0}\circ\sigma^x_{1d_x^0}}(\cdot) 
	\bigg\{\frac{\sum_{j=1}^K \alpha_j \phi_{\mu_{j1d_x^0}^x, s^x_{j1d_x^0} \circ \sigma^x_{1d_x^0}}(\cdot)}
							{\sum_{j=1}^K s_{\cdot j}\alpha_j \phi_{\mu_{j1d_x^0}^x, s^x_{j1d_x^0} \circ \sigma^x_{1d_x^0}}(\cdot)} 
					-1  
	\bigg\} 
}_1\\
& \lesssim \bigg[ \sum_{j=1}^K \abs{\alpha_j - \alpha_j^*} +  \sigma_n^{2\beta} +  \sum_{j=1}^K \alpha_j^*\bigg(\frac{||\mu_j - \mu_j^*||}{\sigma_n}  + \frac{\norm{\beta_{j1d_x^0}}a_{\sigma_n}}{\sigma_n}+ \sigma_n^{2\beta}\bigg)  + \sigma_n^{2\beta} \bigg]  \lesssim  \sigma_n^{2\beta},
\end{align*}
where the penultimate inequality is implied by
$\abs{\sigma_n^2/(s^y_{j} \sigma^y)^2 - 1} \leq  3\sigma_n^{2\beta}$,
$\abs{\sigma_n^2/(s^x_{jk} \sigma^x_k)^2 - 1} \leq  3 \sigma_n^{2\beta}$,
$\abs{s_{\cdot j} - 1} \leq d_x^0 \sigma_n^{2\beta}$, 
$\int ||x_{1d_x^0}|| \phi_{\mu_{j1d_x^0}^x, s^x_{j1d_x^0} \circ \sigma^x_{1d_x^0}}(x_{1d_x^0}) d(x_{1d_x^0}) \lesssim ||\mu_{j1d_x^0}^x|| \leq a_{\sigma_n}$, and
an argument similar to the one preceding \eqref{eq:bd4sigmas}.

Next note that for $\theta \in S_{\theta^\star}$,
\begin{align*}
d_h^2(p(\cdot | \cdot, \theta_{d_x^0},m), p(\cdot | \cdot, \theta,m)) & \lesssim \int \max_{1\leq j \leq m}  
|K_j - \tilde{K}_j|/\abs{K_j} g_0(x) dx + \int \max_{1\leq j \leq m}  \frac{\abs{x'\beta_j}}{\sigma_n} g_0(x) dx\\
& \lesssim \sigma_n^{2\beta} +   \tilde{\epsilon}_n^{2b_1} \int \norm{x} g_0(x) dx \lesssim \sigma_n^{2\beta}, 
\end{align*}
where the first part of the penultimate inequality follows similarly to the proof of Corollary \ref{cr:irr_cov}.

Next, let us bound the ratio  $p(y | x, \theta, m)/f_0(y|x)$.   
For $||(y, x)|| \leq a_{\sigma_n}$,
observe that 
$ \exp \big\{- |y - \mu_j^y - x^\prime\beta_j |^2/(2\sigma_n^2)\big\}  \geq  \exp \big\{- |y - \mu_j^y|^2/\sigma_n^2 - |x^\prime\beta_j |^2/\sigma_n^2\big\}$ and 
$|x^\prime\beta_j| / \sigma_n \leq  a_{\sigma_n}{\epsilon}_n^{2b_1} \leq 1$.  Thus, $\lambda_n$ can be defined by \eqref{eq:lambda_def}.

For $\norm{(y, x)} \geq a_{\sigma_n}$,  
\begin{align*}
\bigg\{\log \frac{f_0(y|x)}{p(y|x, \theta, m)}\bigg\}^2 \lesssim \bigg(\frac{a_{\sigma_n}^4}{\sigma_n^4} + \frac{|y|^4}{\sigma_n^4} +  \norm{x}^4 \tilde{\epsilon}_n^{4b_1}\bigg), 
\end{align*} 
which implies that
\begin{align*}
& \int \bigg\{\log \frac{f_0(y|x)}{p(y|x, \theta, m}\bigg\}^2 1\bigg\{ \frac{p(y|x, \theta, m}{f_0(y|x)}  < \lambda_n\bigg\} f(y|x) g_0(x) dydx 
\\
& \lesssim \bigg[ \frac{a_{\sigma_n}^4 P_0(\norm{Z} > a_{\sigma_n}) }{\sigma_n^4}   +   
\bigg\{\frac{ E_0(|Y|^8)^{1/2} }{\sigma_n^4} +  \tilde{\epsilon}_n^{4b_1}E_0(\norm{X}^8)^{1/2} \bigg\}(P_0(\norm{Z} > a_{\sigma_n}))^{1/2}
\bigg]  \\
& \lesssim \sigma_n^{2\beta + \varepsilon/2},
\end{align*}
as in the proof of Corollary \ref{cr:irr_cov}.  

The lower  bound for the prior probability of $S_{\theta^\star}$ and $m=K$ is the same as the one in 
Theorem \ref{th:prior_thickness}, except $d$ is replaced with $d^0$.  The only additional calculation is as follows,
\begin{equation*}
\Pi(s^y_j,s^x_{jk} \in [1,1+\sigma_n^{2\beta}], j=1,2,\ldots, K;  k=1, \ldots, d_x)
\gtrsim 
\exp \{ - 2 \beta K d \log (1/\sigma_n) \}, 
\end{equation*}
which can be bounded from below as required by the arguments in the proof of Theorem \ref{th:prior_thickness}.
Thus, the prior thickness condition  follows. 

Finally, let us consider bounds on the sieve entropy and the prior probability of the sieve's complement.
The argument here involves only minor changes in the proofs of Theorem \ref{th:sieve} and Corollary \ref{cr:nonc}.
In the definition of sieve 
\eqref{eq:sieve}, let us add the following conditions for the $\beta_j$s and the local scale parameters
\begin{align*}
\beta_j & \in [-\overline{\beta}^{x}, \overline{\beta}^{x}]^{d_x},  \overline{\beta}^{x} =   \overline{\mu}^{x}, j=1, \ldots, m \\
s^y_j,  s^x_{jk} & \in [\underline{\sigma}, \overline{\sigma}],  k=1,\ldots,d_x, \; j =1, 2, \ldots, m. 
\end{align*}

As in Corollary \ref{cr:nonc} and Corollary \ref{cr:irr_cov}, we aim to find the covering number of $\mathcal{F}$ in $d_1$ instead of $d_{SS}$. 
First, let us replace the definition of $S_{\sigma}$ in the proof of Corollary \ref{cr:nonc} with 
\begin{align*}
S_{\sigma} =  \{& \sigma^{l},l=1, \ldots, N_{\sigma} = \lceil{ \log (\overline{\sigma}^2/\underline{\sigma}^2)/ 
(\log (1 + \underline{\sigma}^4\epsilon/(2\cdot 384(\overline{\mu}^x)^2\max\{d_x,d_y\})}\rceil, \\
& \sigma^1= \underline{\sigma},  
(\sigma^{l+1} - \sigma^l)/\sigma^l = \underline{\sigma}^4\epsilon/(2\cdot 384(\overline{\mu}^x)^2\max\{d_x,d_y\})\}
\end{align*}
and
use this $S_{\sigma}$ as the grid for $s^y_j$,  $s^x_{jk}$, $\sigma^y$, and $\sigma^x_k$, $k=1,\ldots,d_x$, $j =1, 2, \ldots, m$.
Note that for $\tilde{\sigma}>\sigma$ and $\tilde{s}>s$, 
$|\sigma s - \tilde{\sigma} \tilde{s}|/(\sigma s) \leq |\sigma - \tilde{\sigma} |/\sigma + |s - \tilde{s}|/s$ and that is why $384$
is replaced by $2\cdot 384$ in the new definition of $S_{\sigma}$.
Since $s^y_j \sigma^y, s^x_{jk}\sigma^x_k \in [\underline{\sigma}^2,\overline{\sigma}^2]$, all the bounds obtained in 
Corollary \ref{cr:nonc} now involve $(\underline{\sigma}^2,\overline{\sigma}^2)$ in place of 
$(\underline{\sigma},\overline{\sigma})$.

Another difference is in the treatment of the new term $x'\beta_j$.  Observe that  for $\beta_j^{(1)}, \beta_j^{(2)} \in \mathcal{F}$ for $j=1, \ldots, m$, 
\begin{eqnarray*}
\int_{\mathcal{X}} \max_{1\leq j \leq m} \abs{x'\beta_j^{(1)} - x'\beta_j^{(2)}} g_0(x) dx \leq  \max_{1\leq j \leq m}\norm{\beta_j^{(1)} - \beta_j^{(2)}} 
\int_{\mathcal{X}} \norm{x}g_0(x) dx .
\end{eqnarray*}
Let us define $S_{\beta}^m$ to  contain centers of $|S_{\beta}^m| =  \lceil 2\cdot 192d_x \int_{\mathcal{X}} \norm{x}g_0(x) dx (\overline{\beta}^x)^2/(\underline{\sigma}^4 \epsilon) \rceil$ equal length intervals partitioning $[-\overline{\beta}, \overline{\beta}]$.
$S_{\mu^x}^m$ now contains centers of $|S_{\mu^x}^m| =  \lceil 2\cdot 192d_x (\overline{\mu}^x)^2/(\underline{\sigma}^2 \epsilon) \rceil$ equal length intervals 
partitioning $[-\overline{\mu}^x, \overline{\mu}^x]$. 

As in the proof of Corollary \ref{cr:nonc}, we thus obtain
\begin{align*}
J(\epsilon, \mathcal{F}, d_1) \leq & 
H \cdot \left \lceil \frac{16\overline{\mu}d_y}{\underline{\sigma}^2\epsilon} \right \rceil^{Hd_y}
\cdot \left \lceil \frac{2\cdot 192 d_x (\overline{\mu}^x)^2}{\underline{\sigma}^4 \epsilon} \right \rceil^{Hd_x}
\cdot H \left \lceil \frac{ \log (\underline{\alpha}^{-1}) }{\log (1 + \epsilon/[12H])} \right \rceil^{H-1}
\\ 
& \cdot \left \lceil \frac{2\cdot 192 d_x \int_{\mathcal{X}} \norm{x}g_0(x) dx \cdot \overline{\beta}^2}{\underline{\sigma}^4 \epsilon} \right \rceil^{Hd_x} \\
& \cdot \left \lceil \frac{ \log (\overline{\sigma}^2/\underline{\sigma}^2)} {\log (1 + \underline{\sigma}^4\epsilon/[2\cdot 384 (\overline{\mu}^x)^2 \max\{d_x,d_y\}])} \right \rceil^{d(H+1)}.
\end{align*}
Observe that $\Pi(\mathcal{F}^c)$ is bounded above by 
\begin{align*}
  & H^2 \exp\{-a_{13} \overline{\mu} ^{\tau_3}\} + H^2 \exp\{-a_{16} (\overline{\mu}^x)^{\tau_5}\}  + H^2 \underline{\alpha}^{a/H} 
  +   \exp\{-a_{10} H(\log H)^{\tau_1}  \} \\
	&
	+ d \Pi (\sigma^y \notin [\underline{\sigma}, \overline{\sigma}])
	+ d H \Pi (s_j^y \notin [\underline{\sigma}, \overline{\sigma}]). 
\end{align*}
The rest of the proof follows the argument in the proof of Theorem \ref{tm:sieve_n}
with the same sequences, except $\underline{\sigma}=n^{-1/a_3}$ (as the prior for $(s_j^y)^2$ satisfies 
the same conditions (\eqref{eq:asnPrior_sigma1}-\eqref{eq:asnPrior_sigma3}) as the prior for $\sigma^y$) and $\overline{\mu}^x = n^{1/\tau_5}$. 
Thus, the claim of Corollary \ref{cr:irr_cov} holds. 
\end{proof}

%\bibliographystyle{ecta}
%\bibliographystyle{model5-names}
%\bibliography{bdp,mixtures}

\end{document}